\documentclass[12pt,reqno]{amsart}

\usepackage{ifthen}

\newcommand{\ifsodaelse}[2]{\ifthenelse{\isundefined{\SODAF}}{#2}{#1}}

\ifsodaelse    {\usepackage{ltexpprt}\usepackage{amsmath}}{\usepackage{fullpage}}
\usepackage{mathrsfs,amssymb,graphicx,verbatim}
\usepackage{paralist}
\usepackage[breaklinks,pdfstartview=FitH]{hyperref}

\newcommand{\mnote}[1]{}

\newcommand{\e}{\varepsilon}
\newcommand{\R}{\mathbb R}
\newcommand{\1}{\mathbf 1}

\ifsodaelse{}{
\newtheorem{theorem}{Theorem}
\newtheorem{proposition}{Proposition}[section]
\newtheorem{lemma}[proposition]{Lemma}
\newtheorem{corollary}[proposition]{Corollary}
}
\newtheorem{claim}[proposition]{Claim}

\ifsodaelse{}
{\theoremstyle{remark}}
\newtheorem{conjecture}{Conjecture}[section]

\newtheorem{remark}[conjecture]{Remark}
\ifsodaelse{}{\theoremstyle{definition}}

\newcommand{\zigzag}{\mathbin{\raisebox{.2ex}{
      \hspace{-.4em}$\bigcirc$\hspace{-.65em}{\rm z}\hspace{.25em}}}}

\newcommand{\A}{\mathscr A}
\newcommand{\C}{\mathscr C}

\DeclareMathOperator*{\Avg}{\mathbb{ E}}
\DeclareMathOperator*{\E}{\mathbb{ E}}
\DeclareMathOperator{\sarkozy}{\A}

\newcommand{\F}{\mathbb F}
\newcommand{\N}{\mathbb N}
\newcommand{\Z}{\mathbb Z}

\newcommand{\oz}{\zigzag}
\newcommand{\circr}{{\bigcirc\!\!\!\!\mathrm{r}\,\,}}

\newcommand{\pconst}{\gamma}
\newcommand{\bpconst}{\gamma_+}

\ifsodaelse{\newcommand{\qed}{\ensuremath{\Box}} \newcommand{\qedhere}{}}{}

\ifsodaelse{
\renewenvironment{proof}[1][\proofname]{\par
  \normalfont \relax
  \trivlist
  \item[\hskip\labelsep
        \itshape 
    #1{.}]\ignorespaces
}{%
  \hfill \qed\endtrivlist
}
\providecommand{\proofname}{Proof}
}
{}

\begin{document}

\ifsodaelse
     {\title{{\Large Towards a Calculus for Non-Linear Spectral Gaps \linebreak {\small{[Extended Abstract]}}}\thanks{Also available (in a different layout) at \url{http://arxiv.org/abs/XXXX.XXXXv1}. Full version will be posted at \url{http://arxiv.org/abs/XXXX.XXXX}.}}}
     {\title[Non-linear spectral gaps]{Towards a calculus for non-linear spectral gaps\\ \small [extended abstract]}\thanks{Extended abstract. To be published (in abridged form)  in the proceedings of the ACM-SIAM Symposium on Discrete Algorithms 2010 (SODA~'10).}}

\ifsodaelse
{\author{Manor Mendel\thanks{The Open University of Israel.} \\
\and
Assaf Naor\thanks{Courant Institute.}}}
{\author{Manor Mendel}
\address {Computer Science Division\\
The Open University of Israel}
\email{mendelma@gmail.com}.
\author{Assaf Naor}
\address{Courant Institute\\ New York University}
\email{naor@cims.nyu.edu}}

\date{}

\ifsodaelse{\maketitle}{}

\begin{abstract}
Given a finite regular graph $G=(V,E)$ and a metric space $(X,d_X)$, let $\gamma_+(G,X)$ denote the smallest constant $\gamma_+>0$ such that for all $f,g:V\to X$ we have:
\begin{equation*}
\frac{1}{|V|^2}\sum_{x,y\in V} d_X(f(x),g(y))^2\le \frac{\gamma_+}{|E|} \sum_{xy\in E} d_X(f(x),g(y))^2.
\end{equation*}
In the special case $X=\R$ this quantity coincides with the reciprocal of the absolute spectral gap of $G$, but for other geometries  the parameter $\gamma_+(G,X)$, which we still think of as measuring the non-linear spectral gap of $G$ with respect to $X$ (even though there is no actual spectrum present here), can behave very differently.

Non-linear spectral gaps  arise often in the theory of metric embeddings, and in the present paper we systematically study the theory of non-linear spectral gaps, partially  in order to obtain a combinatorial construction of super-expander --- a family of bounded-degree graphs $G_i=(V_i,E_i)$, with $\lim_{i\to \infty} |V_i|=\infty$, which do not admit a coarse embedding into any uniformly convex normed space. In addition, the bi-Lipschitz distortion of $G_i$ in any uniformly convex Banach space is $\Omega(\log |V_i|)$, which is the worst possible behavior due to Bourgain's embedding theorem~\cite{Bourgain-embed}. Such remarkable graph families were previously known to exist due to a tour de force algebraic construction of Lafforgue~\cite{Lafforgue}.
Our construction is different and combinatorial,
relying on the zigzag product of Reingold-Vadhan-Wigderson~\cite{RVW}.

We show that non-linear spectral gaps behave sub-multiplicatively under zigzag products --- a fact that amounts to a simple iteration of the inequality above. This yields as a special case a very simple (linear algebra free) proof of the Reingold-Vadhan-Wigderson theorem which states that zigzag products preserve the property of having an absolute spectral gap (with quantitative control on the size of the gap). The zigzag iteration of Reingold-Vadhan-Wigderson also involves taking graph powers, which is trivial to analyze in the classical ``linear" setting. In our work, the behavior of non-linear spectral gaps under graph powers becomes a major geometric obstacle, and we show that for uniformly convex normed spaces there exists a satisfactory substitute for spectral calculus which makes sense in the non-linear setting. These facts, in conjunction with a variant of Ball's notion of Markov cotype and a Fourier analytic proof of the existence of appropriate ``base graphs", are shown to imply that Reingold-Vadhan-Wigderson type constructions can be carried out in the non-linear setting.
\end{abstract}

\ifsodaelse{}{}

\ifsodaelse{}{\maketitle \newpage}


\section{Introduction}

Let $A=(a_{ij})$ be an $n\times n$ symmetric stochastic matrix and
let $$1=\lambda_1(A)\ge \lambda_2(A)\ge \cdots\ge \lambda_n(A)\ge
-1$$ be its eigenvalues. The reciprocal of the spectral gap of $A$,
i.e., the quantity $\frac{1}{1-\lambda_2(A)}$, is the smallest constant $\gamma>0$
such that for every $x_1,\ldots,x_n\in \R$ we have
\begin{equation}\label{eq:R-poin}
\frac{1}{n^2}\sum_{i=1}^n\sum_{j=1}^n(x_i-x_j)^2\le
\frac{\gamma}{n}\sum_{i=1}^n\sum_{j=1}^n a_{ij} (x_i-x_j)^2.
\end{equation}
By summing over the coordinates with respect to some orthonormal
basis,~\eqref{eq:R-poin} can be restated as follows: the value
$\frac{1}{1-\lambda_2(A)}$ is the smallest constant $\gamma>0$ such
that for all $x_1,\ldots,x_n\in L_2$ we have
\begin{equation}\label{eq:L_2-poin} \frac{1}{n^2}\sum_{i=1}^n\sum_{j=1}^n\|x_i-x_j\|_2^2\le
\frac{\gamma}{n}\sum_{i=1}^n\sum_{j=1}^n a_{ij} \|x_i-x_j\|_2^2.
\end{equation}

Once one realizes the validity of inequality~\eqref{eq:L_2-poin} it
is natural to generalize it in several ways. For example, we can
replace the exponent $2$ by some other exponent $p>0$ and, more
crucially, we can replace the Euclidean geometry by some other
metric space $(X,d)$. Such generalizations are standard practice in
metric geometry, as we shall discuss below. For the sake of
presentation, we can take this generalization to even greater
extremes: let $X$ be an arbitrary set and let $K: X\times X\to
[0,\infty)$ be a symmetric function. Such functions are often called
{\em kernels} in the literature, and we shall adopt this terminology
here. Define the reciprocal spectral gap of $A$ with respect to
$K$, denoted $\gamma(A,K)$, as the infimum over all $\gamma\ge
0$ such that for all $x_1,\ldots,x_n\in X$ we have
\begin{equation}\label{eq:kernel-poin} \frac{1}{n^2}\sum_{i=1}^n\sum_{j=1}^nK(x_i,x_j)\le
\frac{\gamma}{n}\sum_{i=1}^n\sum_{j=1}^n a_{ij}K(x_i,x_j).
\end{equation}

In what follows we will also call $\gamma(A,K)$ the Poincar\'e
constant of $A$ with respect to $K$. Readers are encouraged to
focus on the case when $K$ is a power of some metric on $X$, though
as will become clear presently, a surprising amount of ground can be
covered without any assumption on the kernel $K$. For concreteness
we restate the above discussion: the standard gap in the linear
spectrum of $A$ corresponds to considering Poincar\'e constants with
respect to Euclidean spaces (i.e., kernels which are squares of
Euclidean metrics), but there is scope for a theory of non-linear
spectral gaps when one considers inequalities such
as~\eqref{eq:kernel-poin} with respect to other geometries.
The
purpose of this paper is to make some steps towards such a theory,
with emphasis on possible extensions of spectral calculus to non-linear
(non-Euclidean) settings. We apply our new calculus for non-linear
spectral gaps to construct new strong types of expanders, and to resolve a question of Lafforgue~\cite{Lafforgue}. We
obtain a new combinatorial construction of a remarkable type of
bounded degree graphs whose shortest path metric is incompatible with the geometry of any uniformly convex normed space in a very strong sense (i.e., coarse non-embeddability).
The existence of such graph families was first discovered by
Lafforgue~\cite{Lafforgue} via an algebraic construction. Our work indicates that there
is hope for a useful theory of non-linear spectral gaps, beyond the
sporadic examples that have been previously studied in the
literature.

\subsection{Coarse non-embeddability}\label{sec:coarse}
A sequence of metric spaces $\{(X_n,d_n)\}_{n=1}^\infty$ is said to embed coarsely
(with uniform moduli) into a metric space $(Y,d_Y)$ if there exist two non-decreasing functions $\alpha,\beta:[0,\infty)\to [0,\infty)$ such that $\lim_{t\to\infty}\alpha(t)=\infty$, and mappings $f_n:X_n\to Y$, such that for all $n\in\N$ and $x,y\in X_n$ we have:
\begin{equation}\label{eq:coarse condition}
\alpha\left(d_{X_n}(x,y)\right)\le d_Y(f_n(x),f_n(y))\le \beta\left(d_{X_n}(x,y)\right).
\end{equation}
Equation~\eqref{eq:coarse condition} should be viewed as a weak form of ``metric faithfulness" of the mappings $f_n$: this seemingly humble  requirement can be restated informally as ``large distances should map uniformly to large distances". Nevertheless, this weak notion of embedding (which is much weaker than, say, bi-Lipschitz embeddability), has a wide range of applications in geometry and group theory: see for example the book~\cite{Roe03}
and the references therein for (a small part of) such applications.

Since coarse embeddability is a weak requirement, it is quite difficult to prove coarse non-embeddability: very few methods to establish such a result are known, among which is the use of non-linear spectral gaps, as pioneered by Gromov~\cite{Gromov-random-group} (other such methods are coarse notions of metric dimension~\cite{Gro93}, or the use of metric cotype~\cite{MN-cotype}. These methods do not seem to be applicable to the question that we study here). Gromov's argument is simple: assume that $X_n=(V_n,E_n)$ are regular graphs of bounded degree and that $d_{n}(\cdot,\cdot)$ is the shortest-path metric on $X_n$. Assume also that for some $p,\gamma\in (0,\infty)$ we have for all $n\in \N$ and $f:V_n\to Y$:
\begin{equation}
\label{eq:graph poincare}
\frac{1}{|V_n|^2}\sum_{u,v\in V_n}d_Y(f(u),f(v))^p
\ifsodaelse{\\}{}
\le \frac{\gamma}{|E_n|}\sum_{{x,y}\in E_n} d_Y(f(x),f(y))^p.
\end{equation}

A combination of~\eqref{eq:coarse condition} and~\eqref{eq:graph poincare} yields the bound $\frac{1}{|V_n|^2}\sum_{u,v\in V_n}\alpha\left(d_n(u,v)\right)^p\le \gamma\beta(1)^p$. But, since $X_n$ is a bounded degree graph, at least half of the pairs of vertices $u,v\in V_n$ satisfy $d_n(u,v)\ge c\log |V_n|$, where $c$ is a constant which depends on the implied degree bound (but not on $n$). Thus $\alpha(c\log |V_n|)^p\le 2\gamma\beta(1)^p$, and in particular if $\lim_{n\to \infty}|V_n|=\infty$ then we get a contradiction to the assumption $\lim_{t\to\infty}\alpha(t)=\infty$. Observe in passing that this argument also shows that $X_n$ has bi-Lipschitz distortion $\Omega(\log |V_n|)$ in $Y$ --- such an argument was first used by Linial, London and Rabinovich~\cite{LLR} (see also~\cite{Mat97}) to show that Bourgain's embedding theorem~\cite{Bourgain-embed} is asymptotically sharp.

Assumption~\eqref{eq:graph poincare} can be restated as saying that $\gamma(A_n,d_Y^p)\le \gamma$, where $A_n$ is the normalized adjacency matrix of $X_n$. This condition can be viewed to mean that the graphs $\{X_n\}_{n=1}^\infty$ are ``expanders" with respect to $(Y,d_Y)$.
Note that if $Y$ contains at least two points then~\eqref{eq:graph poincare} implies that $\{X_n\}_{n=1}^\infty$ are necessarily also expanders in the classical sense.

The key point in the coarse non-embeddability question is therefore to construct such $\{X_n\}_{n=1}^\infty$ for which we can prove the inequality~\eqref{eq:graph poincare} for non-Hilbertian targets $Y$. This question has been previously investigated by several authors.  Matou\v{s}ek~\cite{Mat97} devised an extrapolation method for Poincar\'e inequalities
\ifsodaelse{}{(see also the description of Matou\v{s}ek's argument  in~\cite{BLMN05})}
which establishes the validity of~\eqref{eq:graph poincare} for every expander when $Y=L_p$. The work of Ozawa~\cite{Ozawa} and Pisier~\cite{pisier-79,pisier-2008} proves~\eqref{eq:graph poincare} for every expander when $Y$ is Banach space which satisfies certain geometric conditions (e.g., $Y$ can be taken to be a Banach lattice with finite cotype). In~\cite{NS-2004,NR-2005} additional results of this type are obtained.

A normed space is called {\em super-reflexive} if it admits an equivalent norm which is uniformly convex. Recall that a normed space $(X,\|\cdot\|_X)$ is called uniformly convex if for every $\e\in (0,1)$  there exists $\delta=\delta_X(\e)>0$ such that for any two vectors $x,y\in X$ with $\|x\|_X=\|y\|_X=1$ and $\|x-y\|_X\ge \e$ we have $\left\|\frac{x+y}{2}\right\|_X\le 1-\delta$ (thus uniform convexity is a uniform version of {\em strict convexity}). The question whether there exists a sequence of arbitrarily large  graphs of bounded degree which do not admit a coarse embedding into any super-reflexive normed space was posed by Kasparov and Yu in~\cite{KY06}, and was solved in the remarkable work of Lafforgue~\cite{Lafforgue} on the strengthened version of property $(T)$ for $SL_3(\mathbb F)$ when $\mathbb F$ is a non-Archimedian local field (thus, for concreteness, Lafforgue's graphs can be obtained as Cayley graphs of finite quotients of $SL_3(\Z_p)$, where $p$ is a prime and $\Z_p$ is the $p$-adic integers).

In this paper we obtain an alternative solution of the Kasparov-Yu problem using a combinatorial construction based on the ``zigzag expanders" of
Reingold, Vadhan, and Wigderson~\cite{RVW}.
More specifically, we construct a family of 9-regular graphs which satisfies~\eqref{eq:graph poincare} for every super-reflexive Banach space $X$ (where $\gamma$ depends only on $X$) --- such graphs are called \emph{super-expanders}.

We state at the outset that it is a major open question whether every expander satisfies~\eqref{eq:graph poincare} for every uniformly convex normed space $X$.
%
It is also unknown whether there exist graph families of bounded degree and logarithmic girth which do not admit a coarse embedding into any super-reflexive normed space---this question is of particular interest in the context of the potential application to the Novikov conjecture that was suggested by Kasparov and Yu in~\cite{KY06}. Note that some geometric restriction on the target space $X$ must be imposed, since the relation between non-linear spectral gaps and coarse non-embeddability, in conjunction with the fact that every finite metric space embeds isometrically into $\ell_\infty$, shows that (for example) $X=\ell_\infty$ can never satisfy~\eqref{eq:graph poincare} for bounded  degree
family of graphs.

Our combinatorial approach can be used to show that there exist bounded degree graph sequences which do not admit a coarse embedding into any $K$-convex normed space. A normed space $X$ is $K$-convex\footnote{$K$-convexity is also equivalent to $X$ having type strictly bigger than $1$, see~\ifsodaelse{\cite{MS,Mau03}}{\cite{MS,Mau03}}. The
$K$-convexity property is strictly weaker than super-reflexivity,
see~\ifsodaelse{\cite{Jam78}}{\cite{Jam74,JL75,Jam78,PX87}}.} if there exists $\e_0>0$ and $n_0\in \N$ such that any embedding of $\ell_1^{n_0}$ into $X$ incurs distortion at least $1+\e_0$, see~\cite{Pisier-K-convex}. 
 This question was asked by Lafforgue~\cite{Lafforgue}. Recently, independently of our work,
Lafforgue~\cite{lafforgue-2009} managed to modify his argument to obtain coarse non-embeddability into $K$-convex spaces for his graph sequences as well.

\subsection*{Acknowledgments}

Michael Langberg was involved in early discussions on the analysis of the zigzag product.
Keith Ball helped in simplifying this analysis. M. M. was partially supported by ISF grant no. 221/07,
BSF grant no. 2006009, and
a gift from Cisco research center. A. N. was supported in part by NSF grants CCF-0635078 and CCF-0832795, BSF grant 2006009, and the Packard Foundation.




\section{A combinatorial approach to the existence of Super-Expanders}

\subsection{The compatibility of non-linear spectral gaps with
combinatorial constructions}
The parameter $\gamma(A,K)$ will reappear presently, but for the purpose of this section we need to study a variant of it which corresponds to the absolute spectral gap of a matrix (similar to the role of absolute spectral gaps in the work of Reingold-Vadhan-Wigderson~\cite{RVW}). Define $\lambda(A)=\max_{2\le i\le n} |\lambda_i(A)|$ and call the quantity $1-\lambda(A)$ the absolute spectral gap of $A$. Similarly to~\eqref{eq:L_2-poin}, the reciprocal of the absolute spectral gap of $A$ is the smallest constant $\gamma_+>0$ such
that for all $x_1,\ldots,x_n,y_1,\ldots,y_n\in L_2$ we have
\begin{equation}\label{eq:L_2-poin+} \frac{1}{n^2}\sum_{i=1}^n\sum_{j=1}^n\|x_i-y_j\|_2^2\le
\frac{\gamma_+}{n}\sum_{i=1}^n\sum_{j=1}^n a_{ij} \|x_i-y_j\|_2^2.
\end{equation}
Analogously to~\eqref{eq:kernel-poin}, given a kernel $K:X\times X\to [0,\infty)$ we can then define $\gamma_+(A,K)$ to be the the infimum over all $\gamma_+\ge
0$ such that for all $x_1,\ldots,x_n,y_1,\ldots,y_n\in X$ we have
\begin{equation}\label{eq:kernel-poin+} \frac{1}{n^2}\sum_{i=1}^n\sum_{j=1}^nK(x_i,y_j)\le
\frac{\gamma_+}{n}\sum_{i=1}^n\sum_{j=1}^n a_{ij}K(x_i,y_j).
\end{equation}
Note that clearly $\gamma_+(A,K)\ge \gamma(A,K)$.

In what follows we will often deal with regular graphs, which will always be allowed to have self loops and multiple edges. We will use the convention that each self loop contribute $1$ to the degree of a vertex. The normalized adjacency matrix of a $d$-regular graph $G=(V,E)$, denoted $A_G$, is defined as usual by letting its $u,v$ entry be the number of edges joining $u,v\in V$ divided by $d$. When discussing Poincar\'e constants we will interchangeably identify $G$ with $A_G$. Thus, for example, we write $\gamma_+(G,K)=\gamma_+(A_G,K)$.

The starting point of our work is an investigation of the behavior of the quantity $\gamma_+(G,K)$ under certain graph operations, the most important of which (for our purposes) is the zigzag product of Reingold-Vadhan-Wigderson~\cite{RVW}. As we shall see below, combinatorial constructions seem to be well-adapted to controlling non-linear quantities such as $\gamma_+(G,K)$. This crucial fact allows us to use them in a perhaps unexpected geometric context.


Assume now that $G_1=(V_1,E_1)$ is an $n_1$-vertex graph which is $d_1$-regular and that $G_2=(V_2,E_2)$ is a $d_1$-vertex graph which is $d_2$-regular. Since the number of vertices in $G_2$ is the same as the degree in $G_1$, we can identify $V_2$ with the edges emanating from a given vertex $u\in V_1$. Formally, we fix for every $u\in V_1$ a bijection $\pi_u:\{e\in E_1: \; u \in e\} \to V_2$.
Moreover, we fix for every $a\in V_2$ a bijection between $[d_2]=\{1,\ldots,d_2\}$ and the multiset of the vertices adjacent to $a$ in $G_2$,
$\kappa_a:[d_2]\to \{b\in V_2:\; \{a,b\}\in E_2\}$.

The zigzag product $G_1\oz G_2$ is the graph whose vertices are $V_1\times V_2$ and $(u,a),(v,b)\in V_1\times V_2$ are joined by an edge if and only if there exist $i,j\in [d_2]$ such that:
\[ \{u,v\}\in E_1 \ifsodaelse{,\ }{\quad\text{and}\quad} a=\kappa_{\pi_u(\{u,v\})}(i)
\ifsodaelse{,\ }{\quad\text{and}\quad} b=\kappa_{\pi_v(\{u,v\})}(j) .\]
The schematic description of this construction is as follows: think of the vertex set of $G_1\oz G_2$ as a disjoint union of ``clouds" which are copies of $V_2=\{1,\ldots,d_1\}$ indexed by $V_1$. Thus $(u,a)$ is the point indexed by $a$ in the cloud labeled by $u$. Every edge $\{(u,a),(v,b)\}$ of $G_1\oz G_2$ is the result of a three step walk: a ``zig" step in $G_2$ from $a$ to $\pi_u(\{u,v\})$  in $u$'s cloud, a ``zag" step in $G_1$ from $u$'s cloud to $v$'s cloud along the edge $\{u,v\}$ and a final ``zig" step in $G_2$ from $\pi_v(\{u,v\})$ to $b$ in $v$'s cloud.  The  zigzag product is illustrated in Figure~\ref{fig:zigzag-op}. The number of vertices of $G_1\oz G_2$ is $n_1d_1$ and its degree is $d_2^2$.

\begin{figure}[ht]
\begin{center}
\includegraphics{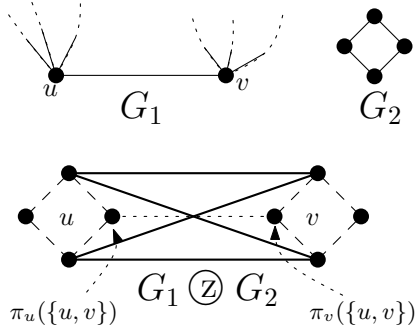}%
\caption{
Illustration of the zigzag product. The upper part of the figure depicts \emph{part} of a 4-regular graph  $G_1$,
and a 4-vertex cycle $G_2$. The bottom part of the figure depicts the edges of  the zigzag product 
between $u$'s cloud and $v$'s cloud.
The original edges of $G_1$ and $G_2$ are drawn as dotted and dashed lines (respectively).}
\label{fig:zigzag-op}
\end{center}
\end{figure}

The zigzag product depends on the labels $\{\pi_u\}_{u\in V_1}$,%
 and in fact different labels of the same graphs can produce non-isomorphic products\footnote{The labels $\{\kappa_a\}_{a\in V_2}$ do not affect the structure of the zigzag product but they are useful in the subsequent analysis.}. However,
all of our results below will be independent of the actual choice of the labeling, so while our notation should formally depend on the labeling, we will drop its explicit mention for the sake of simplicity.

Let us now examine how  $\gamma_+(G_1\oz G_2,K)$ is related to $\gamma_+(G_1,K)$, and $\gamma_+(G_2,K)$, where $K:X\times X\to [0,\infty)$ is an arbitrary kernel. To this end take $f,g:V_1\times V_2\to X$ and note that the definition of $\gamma_+(G_1,K)$ implies that for all $a,b\in V_2$ we have:
\begin{equation}\label{zigzag1-b}
\frac{1}{n_1^2}\sum_{u,v\in V_1}K\left(f(u,a),g(v,b)\right)
\ifsodaelse{\\}{} \le \frac{\gamma_+(G_1,K)}{n_1d_1}\sum_{\{u,v\}\in E_1} K\left(f(u,a),g\left(v,b\right)\right).
\end{equation}
Thus:
\begin{multline}\label{zigzag2-b}
\frac{1}{|V_1\times V_2|^2}\sum_{(u,a),(v,b)\in V_1\times V_2}K(f(u,a),g(v,b))
\ifsodaelse{\\}{} =
\frac{1}{d_1^2}\sum_{a,b\in V_2}\frac{1}{n_1^2}\sum_{u,v\in V_2}K(f(u,a),g(v,b))\\
\stackrel{\eqref{zigzag1-b}}{\le}
\frac{\gamma_+(G_1,K)}{n_1d_1^3}\sum_{a,b\in V_2}\sum_{\{u,v\}\in E_1}K\left(f(u,a),g\left(v,b\right)\right).
\end{multline}
Next, the definition of $\gamma_+(G_2,K)$ implies that for all $u\in V_1$ and $b\in V_2$ we have
\begin{multline}\label{zigzag3-b}
\frac{1}{d_1^2}\sum_{a\in V_2} \sum_{\substack{v\in V_1:\\ \{u,v\}\in E_1}} K\left(f(u,a),g\left(v,b\right)\right)
\\ \le
\frac{\gamma_+(G_2,K)}{d_1d_2}
\sum_{\substack{v\in V_1:\\ \{u,v\}\in E_1}}\sum_{i\in [d_2] } K\left(f\left(u,\kappa_{\pi_u(\{u,v\})}(i)\right),g\left(v,b\right)\right).
\end{multline}

Summing~\eqref{zigzag3-b} over $u\in V_1$ and $b\in V_2$ and plugging it into~\eqref{zigzag2-b}, yields the bound:
\begin{multline}\label{zigzag4-b}
\frac{1}{|V_1\times V_2|^2}\sum_{(u,a),(v,b)\in V_1\times V_2}K(f(u,a),g(v,b))\\\le \frac{\gamma_+(G_1,K)\gamma_+(G_2,K)}{n_1d_1^2d_2}\sum_{v\in V_1}\sum_{i\in[d_2]}\sum_{\substack{u\in V_1:\\ \{u,v\}\in E_1}} \sum_{b\in V_2}
\ifsodaelse{\\}{}
K\left(f\left(u,\kappa_{\pi_u(\{u,v\})}(i)\right),g\left(v,b\right)\right).
\end{multline}
Another application of the definition of $\gamma_+(G_2,K)$ implies that for all $v\in V_1$ and $i\in [d_2] $ we have:
\begin{multline}\label{zigzag5-b}
\frac{1}{d_1^2}\sum_{\substack{u\in V_1:\\ \{u,v\}\in E_1}} \sum_{b\in V_2}
K\left(f\left(u,\kappa_{\pi_u(\{u,v\})}(i)\right),g\left(v,b\right)\right)
\\
\ifsodaelse{\shoveleft{\le}}{\le}
 \frac{\gamma_+(G_2,K)}{d_1d_2}
\sum_{\substack{u\in V_1:\\ \{u,v\}\in E_1}} \sum_{j\in [d_2]}
\ifsodaelse{\\}{}
K\left(f\left(u,\kappa_{\pi_u(\{u,v\})}(i)\right),g\left(v,\kappa_{\pi_v(\{u,v\})}(j)\right)\right).
\end{multline}
Summing~\eqref{zigzag5-b} over $v\in V_1$ and $i\in [d_2] $ and combining the resulting inequality with~\eqref{zigzag4-b} yields the bound:
\begin{eqnarray*} 
&&\!\!\!\!\!\!\!\!\!\!\!\!\!\!\!\!}{\!\!\!\!\!\!\!\!\!\frac{1}{|V_1\times V_2|^2}\sum_{(u,a),(v,b)\in V_1\times V_2}K(f(u,a),g(v,b))\\
&\ifsodaelse{\shoveleft{\le}}{\le}&
\ifsodaelse{\tfrac}{\frac}{\gamma_+(G_1,K)\gamma_+(G_2,K)^2}{n_1d_1d_2^2}
\ifsodaelse{\!\!\!\!}{}\sum_{\{u,v\}\in E_1} \sum_{i,j\in [d_2]}
\ifsodaelse{ \\ \shoveright{
              K\left(f\left(u,\kappa_{\pi_u(\{u,v\})}(i)\right),
              g\left(v,\kappa_{\pi_v(\{u,v\})}(j)\right)\right)
           }}
           {
              K\left(f\left(u,\kappa_{\pi_u(\{u,v\})}(i)\right),g
              \left(v,\kappa_{\pi_v(\{u,v\})}(j)\right)\right)
           }
\\ &\ifsodaelse{\shoveleft{=}}{=}&
\ifsodaelse{\tfrac}{\frac}{\gamma_+(G_1,K)\gamma_+(G_2,K)^2}{n_1d_1d_2^2}
\ifsodaelse{\!\!\!\!\!\!\!\!\!\!\!\!\!\!\!\!}{\!\!\!\!\!\!\!\!\!}
\sum_{\{(u,a),(v,b)\}\in E(G_1\oz G_2)}
\ifsodaelse{\!\!\!\!\!\!\!\!\!\!\!\!\!\!\!\!\!\!}{\!\!\!\!\!\!\!}
K\left(f\left(u,a\right),g\left(v,b\right)\right).
\end{eqnarray*}


Hence we deduce the following theorem: 

\begin{theorem}[Sub-multiplicativity]\label{thm:sub} Let $G_1=(V_1,E_1)$ be an $n_1$-vertex graph which is $d_1$-regular and let $G_2=(V_2,E_2)$ be a $d_1$-vertex graph which is $d_2$-regular. Then for every every kernel $K:X\times X\to [0,\infty)$,
$$
\gamma_+\left(G_1\oz G_2,K\right)\le \gamma_+(G_1,K)\cdot \gamma_+(G_2,K)^2.
$$
\end{theorem}

In the special case $X=\R$ and $K(x,y)=(x-y)^2$, Theorem~\ref{thm:sub} becomes:
\begin{equation}\label{eq:RVW case}
\frac{1}{1-\lambda(G_1\oz G_2)}\le \frac{1}{1-\lambda(G_1)}\cdot \frac{1}{(1-\lambda(G_2))^2}.
\end{equation}
Thus $\lambda(G_1\oz G_2)\le f(\lambda(G_1),\lambda(G_2))$ where $f(\lambda_1,\lambda_2)<1$ when $\lambda_1,\lambda_2\in (0,1)$. This is the main result of Reingold-Vadhan-Wigderson~\cite{RVW}, and it coincides with the later bound of Reingold-Trevisan-Vadhan~\cite{RTV06}. We note that in~\cite{RVW} an improved bound for $\lambda(G_1\oz G_2)$ is obtained which is better than the bound of~\cite{RTV06} (and hence also~\eqref{eq:RVW case}), though this improvement has not been used (so far) in the literature. Theorem~\ref{thm:sub} shows that the fact that zigzag products preserve spectral gaps  has nothing to do with the underlying Euclidean geometry (or linear algebra) that was used in~\cite{RVW,RTV06}: this is a truly non-linear phenomenon which holds in much greater generality, and simply amounts to an iteration of the Poincar\'e inequality~\eqref{eq:kernel-poin+}. In the full version of this paper we present versions of the  sub-multiplicativity result in Theorem~\ref{thm:sub} for other types of graph products such as tensor products, replacement products~\ifsodaelse{\cite{RVW}}{\cite{Gromov-filling,RVW}}, \ifsodaelse{and}{} balanced replacement products~\cite{RVW}\ifsodaelse{.}{, and derandomized square~\cite{RV05}.}


\subsection{An iterative construction of super-expanders}


Let us briefly recall how the Euclidean case of Theorem~\ref{thm:sub} was used by Reingold-Vadhan-Wigderson~\cite{RVW} to construct expanders%
\ifsodaelse{}{ (see also the nice exposition in Section~9.2 of the survey~\cite{HLW})}.

For a graph $G=(V,E)$ and for $t\in \N$, let $G^t$ be the graph in which an edge between $u,v\in V$ is drawn for every walk in $G$ of length $t$ whose endpoints are $u,v$. Thus $A_{G^t}=(A_G)^t$ and if $G$ is $d$-regular then $G^t$ is $d^t$-regular.

Let $H$ be an arbitrary regular graph with $n_0$ vertices and degree $d_0$. Such a graph $H$ will be called the {\em base graph} in what follows. Fix $t_0\in \N$. Define $G_1=H^{2}$ and inductively  $G_{i+1}=G_i^{t_0}\oz H$ (in~\cite{RVW} it sufficed to take $t_0=2$, but we will quickly realize that in the non-linear setting we need to work with higher powers, which is the reason why we are stating the construction for general $t_0$). For this construction to be feasible we need to ensure that $n_0=d_0^{2t_0}$, since in that case for all $i\in \N$ the graph $G_i$ is well defined and has $n_0^{i}=d_0^{2it_0}$ vertices and degree $d_0^2$. Using the fact that $\lambda(G^{t_0})=\lambda(G)^{t_0}$, we see that $\lambda(G_1)=\lambda(H)^2$ and it follows from~\eqref{eq:RVW case} that for $i\in \N$ we have $\lambda(G_{i+1})\le 1-\left(1-\lambda(G_i)^{t_0}\right)(1-\lambda(H))^2$. We thus see inductively that if $\lambda(H)\le 1-(2-2^{1-t_0})^{-1/2}$ then $\lambda(G_i)\le \frac12$ for all $i\in \N$.

Given Theorem~\ref{thm:sub}, the above scheme suggests that we can repeat the Reingold-Vadhan-Wigderson iteration for every kernel $K$ for which a sufficiently good base graph $H$ exists. There is, however, an obstacle to this approach: we used above the identity $\lambda(A^t)=\lambda(A)^t$ ($t\in \N$) in order to increase the spectral gap of $G_i$ in each step of the iteration. While this identity is a trivial corollary of spectral calculus, and was thus the ``trivial part" of the construction in~\cite{RVW}, there is no reason to expect that $\gamma_+(A^t,K)$ decreases similarly with $t$ for other kernels $K:X\times X\to [0,\infty)$. To better grasp what is happening here let us examine the asymptotic behavior of $\gamma_+(A^t,|\cdot|^2)$ as a function of $t$ (here and in what follows $|\cdot |$ denotes the absolute value on $\R$):
\begin{multline}\label{eq:euclidean decay}
\gamma_+\left(A^t,|\cdot|^2\right)=\frac{1}{1-\lambda(A^t)}=\frac{1}{1-\lambda(A)^t}
\\
=\frac{1}{1-\left(1-\frac{1}{\gamma_+\left(A,|\cdot|^2\right)}\right)^t}\asymp \max\left\{1,\frac{\gamma_+\left(A,|\cdot|^2\right)}{t}\right\},
\end{multline}
where above, and in what follows, $\asymp$ denotes equivalence up to universal constants (we will also use the notation $\lesssim,\gtrsim$ to express the corresponding inequalities up to universal constants). Restating~\eqref{eq:euclidean decay} in words, raising a matrix to a large power $t\in \N$ corresponds to decreasing its (real) Poincar\'e constant by a factor of $t$ as long as it is possible to do so (note that the Poincar\'e constant is necessarily at least $1$).

For our scheme to work for other kernels $K:X\times X\to [0,\infty)$ we would like $K$ to satisfy a  ``spectral calculus" inequality of this type, i.e., an inequality which ensures that, if $\gamma_+(A,K)$ is large, then $\gamma_+(A^t,K)$ is much smaller than $\gamma_+(A,K)$ for sufficiently large $t\in \N$.
This is, in fact, not the case in general: in the Section~\ref{sec:no-decay}
we construct  a metric space $(X,d_X)$ such that for each $n\in \N$ there is a symmetric stochastic matrix $A_n$ such that $\gamma_+(A_n,d_X^2)\ge n$ yet for every $t\in \N$ there is $n_0\in \N$ such that for all $n\ge n_0$ we have $\gamma_+(A_n^t,d_X^2)\ge \frac12 \gamma_+(A_n,d_X^2)$. The question which metric spaces satisfy the required spectral calculus inequalities thus becomes a subtle issue which we believe is of fundamental importance, beyond the particular application that we present here. A large part of the present paper is devoted to addressing this question. We obtain rather satisfactory results which allow us to carry out a zigzag type construction of super-expanders --- expanders with respect to super-reflexive spaces --- though we are still quite far from a complete understanding of the behavior of non-linear spectral gaps under graph powers for non-Euclidean geometries.

We will introduce in Section~\ref{cesaro} a criterion on a metric space $(X,d_X)$, which is a bi-Lipschitz invariant, and we prove that it implies that for every $n,m\in\N$ and every $n\times n$ symmetric stochastic matrix $A$ the  Ces\`aro averages $\frac{1}{m}\sum_{t=0}^{m-1}A^t$ satisfy the following spectral calculus inequality:
\begin{equation}\label{eq:decay for gamma_+}
\gamma_+\left(\frac{1}{m}\sum_{t=0}^{m-1}A^t,d_X^2\right)\le C(X)\max\left\{1,\frac{\gamma_+\left(A,d_X^2\right)}{m^{\e(X)}}\right\},
\end{equation}
where $C(X),\e(X)\in (0,\infty)$ depend only on the geometry of $X$ but not on $n,m,A$. The fact that we can only prove such an inequality for Ces\`aro averages rather than powers does not create any difficulty in the ensuing argument, as we shall see presently. We postpone the discussion of the metric criterion which implies~\eqref{eq:decay for gamma_+} to Section~\ref{cesaro}. It suffices to say here that we show that certain classes of metric spaces, including super-reflexive normed spaces, satisfy this criterion, and hence also~\eqref{eq:decay for gamma_+}.

Note that Ces\`aro averages have the following combinatorial interpretation in the case of graphs: given an $n$-vertex $d$-regular graph $G=(V,E)$ let $\A_m(G)$ be the graph whose vertex set is $V$ and for every $t\in \{0,\ldots,m-1\}$ and $u,v\in V$ we draw $d^{m-1-t}$ edges joining $u,v$ for every walk in $G$ of length $t$ which starts at $u$ and terminates at $v$. With this definition $A_{\A_m(G)}=\frac{1}{m}\sum_{t=0}^{m-1}A_G^t$, and $\A_m(G)$ is $md^{m-1}$-regular. 
We will slightly abuse this  notation by also using the shorthand
$\A_m(A)=\frac 1m \sum_{t=0}^{m-1}A^t$, when $A$ is a matrix.

Another issue which we must overcome is the existence of a sufficiently good base graph $H$ for our zigzag iteration. This is also a subtle point which is discussed in Section~\ref{sec:base} below. In order to present our construction here, we first state the following lemma, and postpone the discussion about it to Section~\ref{sec:base}:

\begin{lemma}[Existence of base graphs]\label{lem:base-intro}
There exists an unbounded increasing sequence of integers $\{n_i\}_{i=1}^\infty$ satisfying $n_{i+1}\le 100n_i$ with the following property: for every $\delta\in (0,1)$ there is a sequence of regular graphs $\{H_i(\delta)\}_{i=1}^\infty$ such that $H_i(\delta)$ has $n_i$ vertices and degree $d_i(\delta)\le e^{(\log n_i)^{1-\delta}}$. Moreover, for every super-reflexive (even just $K$-convex) normed space $(X,\|\cdot\|_X)$ there exist $\delta_0(X),\gamma(X)\in (0,1)$ such that for all $0<\delta\le \delta_0(X)$ and all $i\in \N$, we have $\gamma_+\left(H_i(\delta),\|\cdot\|_X^2\right)\le \gamma(X)$.
\end{lemma}

We briefly discuss a simple operation, called {\em edge completion}, which allows us to change the degree of a graph while not changing its Poincar\'e constant by too much. This operation will generate a lot of freedom in applying zigzag products, since we will be able to apply it in order to adjust the degrees of graphs so that their zigzag product becomes well-defined. Given a $d$-regular graph $G=(V,E)$ and an integer $d'\ge d$ define the $d'$-edge completion of $G$, denoted $\C_{d'}(G)=(V,E')$, as follows. Writing $d'=\ell d+r$ where $\ell\in \mathbb N$, and $r\in \{0,\ldots,d-1\}$, duplicate each edge in $G$ $\ell$ times, and add $r$ self-loops to each vertex of $G$. This makes $\C_{d'}(G)$  a $d'$-regular graph. It is evident from the definition of $\gamma_+(\cdot,\cdot)$ that for every kernel $K:X\times X\to [0,\infty)$ we have $\gamma_+\left(\C_{d'}(G),K\right)\le 2\gamma_+(G,K)$.

With these facts at hand, we can now construct an expander with respect to a given super-reflexive Banach space $X$.

\begin{lemma} \label{thm:expander-wrt-X}
For every super-reflexive Banach space $X$ there exist $C'(X)\ge 1$,
$d=d(X)\in \mathbb N$ and sequence of $d$-regular graphs
$\left\{F_j(X)\right\}_{j=1}^\infty$ whose number of vertices tend
to $\infty$ with $j$, such that  $\bpconst(F_j(X),\|\cdot\|_X^2)\le
C'(X)$ for every $j\in \N$.
\end{lemma}
\begin{proof}
Fix $k\in \N$ which will be determined presently. Write $t_0=(2k)^{3k}$, and let $i_0$ be the smallest $i\in \N$ such that $n_i\ge e^{(4t_0)^k}$. Note that this implies that $n_{i_0}\le 100 e^{(4t_0)^k}$. We shall now define inductively a sequence of graphs $\{F_j(k)\}_{j=0}^\infty$ as follows: $F_0(k)=H_{i_0}(1/k)$ (from Lemma~\ref{lem:base-intro}). The degree of $F_0(k)$ is $d_{i_0}\le e^{(\log n_{i_0})^{1-1/k}}$ and it has $n_{i_0}$ vertices. Define $F_1(k)=\C_{d_{i_0}^2}(F_0(k))$, so that the degree of $F_1(k)$ is $d_{i_0}^2$. Assume inductively that we defined $F_j(k)$ to be a graph with $n_{i_0}^j$ vertices and degree $d_{i_0}^2$. Then $\A_{t_0}(F_j(k))$ has $n_{i_0}^j$ vertices and degree $t_0d_{i_0}^{2(t_0-1)}\le n_{i_0}$ (where we used our choice of $i_0$). It follows the we can form the edge completion $\C_{n_{i_0}}(\A_{t_0}(F_j(k)))$, which has degree $n_{i_0}$, and therefore it is possible to form the zigzag product
$$
F_{j+1}(k)=\left(\C_{n_0}(\A_{t_0}(F_j(k)))\right)\oz F_0(k),
$$
which has degree $d_{i_0}^2$ and $n_{i_0}^{j+1}$ vertices.

Let $(X,\|\cdot\|_X)$ be a super-reflexive normed space and let $\delta_0(X),\gamma(X)$ and $\e(X),C(X)$ be as in Lemma~\ref{lem:base-intro} and~\eqref{eq:decay for gamma_+}, respectively. Assume that $k\ge 1/\delta_0(X)$. Then by Lemma~\ref{lem:base-intro} we have $\gamma_+(F_0(k),\|\cdot\|_X^2)\le \gamma(X)$, and therefore also by the basic property of edge completion, $\gamma_+(F_1(k),\|\cdot\|_X^2)\le 2\gamma(X)$. We can now apply Theorem~\ref{thm:sub} to get the following recursive estimate:
\begin{align*}
\gamma_+ \ifsodaelse{&}{} \left(F_{j+1}(k),\|\cdot\|_X^2\right)
\ifsodaelse{\\}{}
&\le \gamma_+\left(\C_{n_0}(\A_{t_0}(F_j(k))),\|\cdot\|_X^2\right)\gamma_+\left(F_0(k),\|\cdot\|_X^2\right)^2
\\ & \le
2\gamma_+\left(\A_{t_0}(F_j(k)),\|\cdot\|_X^2\right)\gamma(X)^2
\\ & \stackrel{\eqref{eq:decay for gamma_+}}{\le} 2C(X)\gamma(X)^2\max\left\{1,\frac{\gamma_+\left(F_j(k),\|\cdot\|_X^2\right)}{t_0^{\e(X)}}\right\}
\\ &=
2C(X)\gamma(X)^2\max\left\{1,\frac{\gamma_+\left(F_j(k),\|\cdot\|_X^2\right)}{(2k)^{3k\e(X)}}\right\}.
\end{align*}
It follows by induction that if $k\ge \max\left\{C(X),\gamma(X),1/\e(X),1/\delta_0(X)\right\}$ then for all $j\in \N$ we have $\gamma_+\left(F_{j}(k),\|\cdot\|_X^2\right)\le 2C(X)\gamma(X)^2$.
\end{proof}

\mnote{Deleted the union argument of all expanders. It is replaced with a better diagonalization  argument}

Lemma~\ref{thm:expander-wrt-X} is not the final word. Our goal is to
construct a super-expander, i.e., \emph{one} sequence
$\{H_i\}_{i=1}^\infty$ of bounded degree graphs which is expander
family with respect to \emph{all} super-reflexive Banach spaces. In
order to do so, we begin by examining more closely the proof of
Theorem~\ref{thm:expander-wrt-X} and the parameters in
Lemma~\ref{lem:base-intro} and in~\eqref{eq:decay for gamma_+}.


The proof of Theorem~\ref{thm:expander-wrt-X} gives (i) for each $k\in \N$ a sequence of graphs $\left\{F_j(k)\right\}_{j=0}^\infty$, such that $F_j(k)$ has degree $d_k$ and $n_j(k)$ vertices, where $\{n_j(k)\}_{j=0}^\infty$ is a strictly increasing sequence of integers; (ii) a nondecreasing sequence   $\left\{C'_k\right\}_{k=1}^\infty$ such that
for every super-reflexive Banach space $X$, there exists $k_1(X)\in \N$ such that for every integer $k\ge k_1(X)$ and $j\in \N$,
\begin{equation} \label{eq:prop-1}
\gamma_+\left(F_j(k), \|\cdot\|_X^2\right)\le C'_{k_1(X)}.
\end{equation}

From the proof of~\eqref{eq:decay for gamma_+} in
Section~\ref{cesaro} we infer that there exist a nondecreasing
sequence  $\left\{C_k''\right\}_{k=1}^\infty\subseteq (0,\infty)$
and a nonincreasing sequence
$\left\{\e_k\right\}_{k=1}^\infty\subseteq (0,1)$, such that for any
super-reflexive Banach space $X$ there exists $k_2(X)\in \N$ such
that for any regular graph $G$ and $t\in \mathbb N$,
\begin{equation}\label{eq:prop-2}
\bpconst\left(\sarkozy_t(G),\|\cdot\|_X^2\right) \le C_{k_2(X)}'' \cdot \max\left\{ 1, \frac{\bpconst\left(G,\|\cdot\|_X^2\right)}{t^{\e_{k_2(X)}}} \right \} .
\end{equation}
The analogues of properties~\eqref{eq:prop-1} and~\eqref{eq:prop-2}
for a general family of kernels suffice to prove the following
lemma:
\begin{lemma}\label{lem:super-expander}
Let $\left\{F_j(k)\right\}_{j,k\in \N}$ be graphs as described
before~\eqref{eq:prop-1}, i.e., $F_j(k)$ has degree $d_k$ and
$n_j(k)$ vertices, where $\{n_j(k)\}_{j=0}^\infty$ is strictly
increasing. Assume that $\mathscr K$ is a family of kernels such
that $\gamma_+(F_j(k),K)<\infty$ for all $K\in \mathscr K$ and
$j,k\in \N$. Assume also that $\mathscr K$
satisfies~\eqref{eq:prop-1} and~\eqref{eq:prop-2}, i.e.,
\begin{itemize}
\item  There is a nondecreasing sequence   $\left\{C'_k\right\}_{k=1}^\infty$ such that
for every $K\in \mathscr K$ there exists $k_1(K)\in \N$ such that for every integer $k\ge k_1(K)$ and $j\in \N$,
\begin{equation} \label{eq:prop-1'}
\gamma_+\left(F_j(k), K\right)\le C'_{k_1(K)}.
\end{equation}
\item There exist a nondecreasing sequence  $\left\{C_k''\right\}_{k=1}^\infty\subseteq (0,\infty)$ and a nonincreasing sequence $\left\{\e_k\right\}_{k=1}^\infty\subseteq (0,1)$, such that for any $K\in \mathscr K$ there exists
$k_2(K)\in \N$ such that for any regular graph $G$ and $t\in \mathbb N$,
\begin{equation}\label{eq:prop-2'}
\bpconst\left(\sarkozy_t(G),K\right) \le C_{k_2(K)}'' \cdot \max\left\{ 1, \frac{\bpconst\left(G,K\right)}{t^{\e_{k_2(K)}}} \right \} .
\end{equation}
\end{itemize}
Then there exists $d\in \N$ and a sequence of $d$-regular graphs
$\{H_i\}_{i=1}^\infty$  whose number of vertices tends to $\infty$
with $i$, such that for every $K\in \mathscr K$ there exists
$C(K)\in (0,\infty)$ satisfying $\bpconst(H_i,K)\le C(K)$ for all
$i\in \N$.
\end{lemma}

\begin{proof}
Set
$C_k=\max\left\{C'_k,C_k'' \right\}$ and
$m_k=\left\lceil (2C_{k}^3)^{1/\e_{k}} \right\rceil$. Let $j(k)$ be
the smallest $j\in \N$ such that $n_{j}(k) > \max\left\{k,  m_{k+1}
d_{k+1}^{2m_{k+1}}\right\}$. We define $L_k=F_{j(k)}(k)$. While the
sequence $\left\{\gamma_+(L_k,K)\right\}_{k=1}^\infty$ is bounded
for any $K\in \mathscr K$, this is not yet the desired graph
sequence since its degrees are unbounded.

We next define for every $k\in \N$ a sequence of graphs
$L_{k,0},L_{k,1},\ldots, L_{k,\ell(k)}$ as follows. Along with
$\left\{L_{k,j}\right\}_{j=0}^{\ell(k)}$ we shall also define an
auxiliary sequence $\left\{h_j(k)\right\}_{j=0}^{\ell(k)-1}
\subseteq \mathbb N$.
Set $L_{k,0}=L_k$ and $h_0(k)=k$. Define $h_1(k)$ to be the smallest
$h\in \N$ such that $n_{j(h)}(h)> d_{h_0(k)}$. From the definition
of $j(\cdot)$, we have that $h_1(k)< h_0(k)$. Define
$$L_{k,1}=\mathscr A_{m_{h_1(k)}}\left(\mathcal
C_{n_{j(h_1(k))}(h_1(k))}(L_{k,0}) \oz L_{h_1(k)}\right).$$ Note
that the degree of $L_{k,1}$ is $m_{h_1(k)}
d_{h_1(k)}^{2m_{h_1(k)}}$.

Inductively, assume we have already defined $L_{k,i-1}$, and
$h_{i-1}(k)$, and that the degree of $L_{k,i-1}$ is $m_{h_{i-1}(k)}
d_{h_{i-1}(k)}^{2m_{h_{i-1}(k)}}$. Define $h_i(k)$ to be the
smallest $h\in \N$ such that  $n_{j(h)}(h)> m_{h_{i-1}(k)}^2
d_{h_{i-1}(k)}^{2m_{h_{i-1}(k)}}$. From the definition of
$j(\cdot)$, we have that $h_i(k)< h_{i-1}(k)$. Define
\[ L_{k,i}= \A_{m_{h_i(k)}}\left(\C_{n_{j(h_i(k))}(h_i(k))}(L_{k,i-1}) \oz L_{h_i(k)}\right). \]
Note that the degree of $L_{k,i}$ is $m_{h_i(k)}
d_{h_i(k)}^{2m_{h_i(k)}}$.

Continue this way, until $h_i(k)=1$. At that point, set
$\ell(k)=i+1$, and define
\[ L_{k,\ell(k)}=\A_{m_1}(\C_{n_{j(1)}(1)}(L_{k,\ell(k)-1}) \oz L_1). \]
Hence $L_{k,\ell(k)}$ has degree $m_1d_1^{2m_1}$ which is a
universal constant. Define  $H_k=L_{k,\ell(k)}$. We shall now show
that $\{H_k\}_{k=1}^\infty$ is the desired graph sequence.

 Fix $K\in \mathscr K$ and  set $k_0(K)=\max\{k_1(K),k_2(K)\}$.
Now fix $k>k_0(K)$. We shall first estimate
$\gamma_+\left(L_{k,i},K\right)$, when $h_i(k)> k_0(K)$. By
induction on $i$ we shall show that as long as $h_i(k)>k_0(K)$, we
have $\bpconst\left(L_{k,i},K\right)\le C_{k_0(K)}$. If
$h_0(k)=k>k_0(K)\ge k_1(K)$, then by our
assumption~\eqref{eq:prop-1'},
$\bpconst(L_{k,0},K)=\bpconst(F_{j(k)}(k),K)\le C_{k_0(K)}$. Assume
next that $h_i(k)>k_0(K)$, and inductively that
$\bpconst(L_{k,i-1},K) \le C_{k_0(K)}$. Then using
Theorem~\ref{thm:sub} and~\eqref{eq:prop-2'},
 \begin{multline*}
  \bpconst(L_{k,i},K)
\ifsodaelse{\\}{}
  \le C_{k_0(K)} \cdot\max\left \{1, \frac{2\bpconst(L_{k,i-1},K) \cdot
  \bpconst(L_{h_i(k)},K)^2}{m_{h_i(k)}^{\e_{h_i(k)}} }\right\}
\\
 \stackrel{\eqref{eq:prop-1'}}{\le} C_{k_0(K)}
 \max\left\{1,  \frac{2C_{k_0(K)}^3}{ m_{h_i(k)}^{\e_{h_i(k)}}}\right\} \le C_{k_0(K)}.
\end{multline*}

If we let $i_0(k)$ be the largest $i$ such that $h_i(k)>k_0(K)$,
then we have just proved that $\bpconst(L_{k,i_0(k)},K)\le
C_{k_0(K)}$. For $i>i_0(k)$, we use shall only rough bounds:
 \begin{multline}\label{eq:big i}
  \bpconst(L_{k,i},K)
\ifsodaelse{\\}{}
  \le C_{k_0(K)}\cdot \max\left \{1, \frac{2\bpconst(L_{k,i-1},K) \cdot
  \bpconst(L_{h_i(k)},K)^2}{m_{h_i(k)}^{\e_{h_i(k)}} }\right\}\\\le
  2C_{k_0(K)} \bpconst(L_{h_i(k)},K)^2\bpconst(L_{k,i-1},K).
  \end{multline}
By iterating~\eqref{eq:big i} we obtain the crude bound:
\begin{equation} \label{eq:final-super-expander}
   \bpconst(H_k,K)=\bpconst(L_{k,\ell(k)},K)
\ifsodaelse{\\}{}
   \le  (2C_{k_0}(K))^{k_0(K)}
   \prod_{i=1}^{k_0(K)} \bpconst(L_{i},K)^2.
\end{equation}
The right hand side of~\eqref{eq:final-super-expander} depends on
$K$, but not on $k$. Hence  $\{\bpconst(H_k,K)\}_{k=1}^\infty$ is
bounded for every $K\in \mathscr K$. It remains to argue that the
number of vertices of $H_k$ tends to $\infty$ with $k$. Indeed, the
number of vertices in $H_k$ is at least as the number of vertices of
$L_k=F_{j(k)}(k)$ which is $n_{j(k)}(k)$, and from the definition of
$j(k)$, we have $n_{j(k)}(k) >k$.
\end{proof}

\begin{theorem}[Existence of super-expanders] \label{thm:generic-super-expander}
There exists a sequence of $9$-regular graphs
$\{G_i\}_{i=1}^\infty$, whose number of vertices tends to $\infty$
with $i$, such that for any super-reflexive Banach space
$(X,\|\cdot\|_X)$ there is $C(X)\in (0,\infty)$, such that
$\bpconst(G_i,\|\cdot\|_X^2)\le C(X)$ for every $i\in \N$.
\end{theorem}
\begin{proof}
From the discussion before Lemma~\ref{lem:super-expander}, and by
Lemma~\ref{lem:super-expander} itself, we conclude that there exists
a universal constant $d\in\mathbb N$, and $d$-regular super-expander
family $\{H_k\}_{k=1}^\infty$, with $\bpconst(H_k,\|\cdot\|_X^2)\le
C(X)$.

Let $C_m^\circ$ be the $m$-vertex cycle with self-loops. It is an
easy consequence of the triangle inequality that for every metric
space $(Y,d_Y)$ we have $\gamma_+(C_m^\circ,d_Y^2)\le 4m^2$. Thus,
using Theorem~\ref{thm:sub} once more we see that for all $k\in
\mathbb N$ we have
\[ \gamma_+\left(H_k \oz C^\circ_a,\|\cdot\|_X^2\right)\le 16d^4 C(X), \]
which is a constant depending only on $X$ but not on $k$. Hence
$\{H_k \oz C^\circ_a\}_{k=1}^\infty$ is the required sequence of
$9$-regular super-expanders.
\end{proof}

\begin{remark}
Similarly, the sequence $\{(H_k\oz C^o_a)\circr C_9\}_{k=1}^\infty$
is a 3-regular super-expander, where, $C_9$ is the cycle on 9
vertices (without self-loops), and ``$\circr$'' is the replacement
(graph) product (cf.~\cite{RVW}). In the full version of this
extended abstract we analyze the replacement product similarly to
the zigzag product analysis presented here.
\end{remark}

\begin{remark}\label{rem:lafforgue}
Lafforgue~\cite{Lafforgue} asked whether there exists a sequence
of bounded degree graphs $\{G_k\}_{k=1}^\infty$ which does not admit
a coarse embedding (with uniform moduli) into any $K$-convex Banach
space. An examination of Lafforgue's argument
shows the his method produces graphs $\{H_j(k)\}_{j,k\in \N}$ such
that for each $k\in \N$ the graphs $\{H_j(k)\}_{j\in \N}$ have
degree $d_k$, their cardinalities are unbounded, and for every
$K$-convex Banach space $(X,\|\cdot\|_X^2)$ there is some $k\in \N$
for which $\sup_{j\in \N}\gamma_+(H_j(k),\|\cdot\|_X^2)<\infty$. The
problem is that the degrees $\{d_k\}_{k\in \N}$ are unbounded, but
this can be overcome as above by applying the zigzag product with a
cycle with self-loops. Indeed, define $G_j(k)=H_j(k)\oz
C_{d_k}^\circ$. Then $G_j(k)$ is $9$-regular, and as argued in the
proof of Theorem~\ref{thm:generic-super-expander}, we still have
$\sup_{j\in \N}\gamma_+(G_j(k),\|\cdot\|_X^2)<\infty$. To get a
single sequence of graphs which does not admit a coarse embedding
 into any $K$-convex Banach space, fix a bijection $\psi=(a,b):\N\to \N\times \N$, and define
$G_m=G_{a(m)}(b(m))$. The graphs $G_m$ all have degree $9$. If $X$
is $K$-convex then choose $k\in \N$ as above. If we let $m_j\in \N$
be such that $\psi(m_j)=(j,k)$ then we have shown that the graphs
$\{G_{m_j}\}_{j=1}^\infty$ are arbitrarily large, have bounded
degree, and  satisfy $\sup_{j\in
\N}\gamma_+(G_{m_j},\|\cdot\|_X^2)<\infty$. The argument that was
presented in Section~\ref{sec:coarse} implies that
$\{G_m\}_{m=1}^\infty$ do not embed coarsely into $X$.
\end{remark}

\subsection{Ces\`aro averages of matrices: coping with new issues that arise from non-linearity}\label{cesaro}

The goal in this section is to prove~\eqref{eq:decay for gamma_+}, i.e., the spectral calculus
inequality for Ces\'aro averages in super-reflexive spaces.

Recall that we use the following notation $\A_m(A) =\frac 1m \sum_{t=0}^{m-1}A^t$. Fix $\e>0$ and $C\in(0,\infty)$. Assume that a metric space $(X,d_X)$ satisfies the following property, which we call {\em metric Markov cotype $\frac{2}{\e}$ with constant $C$ and exponent $2$}: for every $m,n\in \N$, every symmetric stochastic $n\times n$ matrix $A=(a_{ij})$, and every $x_1,\ldots,x_n\in X$ there exist $y_1,\ldots,y_n\in X$ which satisfy the following inequality:
\begin{eqnarray}\label{eq:intro def cotype}
\sum_{i=1}^n d_X(x_i,y_i)^2+m^\e\sum_{i=1}^n\sum_{j=1}^na_{ij}d_X(y_i,y_j)^2
\ifsodaelse{\\}{}
\le C^2\sum_{i=1}^n\sum_{j=1}^n \A_m(A)_{ij}\,d_X(x_i,x_j)^2.
\end{eqnarray}
The origin of this (admittedly cumbersome) name comes from a key {\em linear} property of normed
spaces that was introduced by Ball~\cite{Ball} under the name of {\em Markov cotype} in
his profound work on the Lipschitz extension problem. We will see below that any
super-reflexive normed space has metric Markov cotype $\frac{2}{\e}$ with
constant $C$ and exponent $2$ for some $\e\in (0,1]$ and $C\in (0,\infty)$
(our proof uses Ball's insights in~\cite{Ball}).

Informally,~\eqref{eq:intro def cotype} means that the average
square distance of the edges in an embedding of the graph $\frac 1m
\sum_0^{m-1} A^t$ into $X$ is  larger than the average square
distance of the edges in the graph $A$ by a factor of $m^\e$, in a
different embedding of the graph $A$ into $X$, which is near the
original embedding. This is formalized in the following claim whose
proof is in Section~\ref{sec:details-decay}.

\begin{claim} \label{cl:decay for gamma}
Assuming that $(X,d_X)$ has metric Markov cotype $\frac{2}{\e}$ with constant $C$ and exponent $2$, we have:
\begin{equation}\label{eq:decay for gamma}
\gamma\left(\A_m(A),d_X^2\right)\le 12C^2\max\left\{1,\frac{\gamma\left(A,d_X^2\right)}{m^\e}\right\}.
\end{equation}
\end{claim}

In order to deduce~\eqref{eq:decay for gamma_+}, which is the
variant of~\eqref{eq:decay for gamma} for absolute spectral gaps, we
need the following two easy consequences of the  triangle
inequality. The proofs appear in Section~\ref{sec:details-decay}.
\begin{claim} \label{cl:gamma+-as-tensor}
For every doubly stochastic $n\times n$ matrix $A$, and any metric space $X$,
\begin{equation}\label{eq:2n}
 \frac25 \gamma\left( \left ( \begin{smallmatrix}
  0 & A \\
   A & 0
   \end{smallmatrix} \right ),d_X^2\right)\le \gamma_+\left( A,d_X^2\right)\le 2\gamma\left(
   \left(\begin{smallmatrix}
  0 & A \\
   A & 0
   \end{smallmatrix} \right ),d_X^2\right).
   \end{equation}
\end{claim}

\begin{claim} \label{cl:commute tenzor}
For every doubly stochastic $n\times n$ matrix $A$, and any metric space $X$,
\begin{equation}\label{eq:commute tenzor}
\gamma\left( \left( \begin{smallmatrix}
  0 & \A_m(A) \\
   \A_m(A) & 0
   \end{smallmatrix} \right )
,d_X^2\right)\le 9 \gamma\left( \A_m \left (\begin{smallmatrix}
  0 & A \\
   A & 0
   \end{smallmatrix} \right),d_X^2\right).
\end{equation}
\end{claim}

By combining~\eqref{eq:decay for gamma} with~\eqref{eq:2n}
and~\eqref{eq:commute tenzor} we deduce~\eqref{eq:decay for gamma_+}
with $C(X)\le 540C^2$ whenever $(X,d_X)$ satisfies~\eqref{eq:intro
def cotype}.

\medskip

It remains to explain why any super-reflexive normed space
$(X,\|\cdot\|_X)$ has metric Markov cotype $\frac{2}{\e}$ with
constant $C$ and exponent $2$ for some $\e,C\in (0,\infty)$. An
important theorem of Pisier~\cite{Pisier-martingales} says that $X$
has an equivalent norm whose modulus of uniform convexity has power
type $p$ for some $p\in [2,\infty)$. From this  fact, together with
results from~\cite{Fiegel76,BCL}, we deduce the following variant of
Pisier's martingale inequality~\cite{Pisier-martingales}, which is
proved in Section~\ref{sec:proof-pisier-variant}:
\begin{equation}\label{eq:pisier variant}
\sum_{k=1}^m \E \left[\|M_k-M_{k-1}\|_X^2\right]\lesssim
m^{1-\frac{2}{p}}\E \left[\|M_m-M_{0}\|_X^2\right],
\end{equation}
for every square integrable martingale $\{M_k\}_{k=0}^m\subseteq X$,
where the implied constant in~\eqref{eq:pisier variant} depends on
$X$.

In order to deduce~\eqref{eq:intro def cotype} from~\eqref{eq:pisier variant} fix $x_1,\ldots,x_n\in X$ and define $f:\{1,\ldots,n\}\to X$ by $f(i)=x_i$.
For every $\ell\in \{1,\ldots,n\}$ let
$\bigl\{Z_t^{(\ell)}\bigr\}_{t=0}^m$ be the Markov chain on
$\{1,\ldots,n\}$ which starts at $\ell$ and has transition matrix $A$. In other
words $Z_0^{(\ell)}=\ell$ with probability $1$ and for $t\in
\{1,\ldots,m\},\ i,j\in \{1,\ldots,n\}$ we have
$\Pr\bigl[\bigl.Z_t^{(\ell)}=j\bigr|Z_{t-1}^{(\ell)}=i\bigr]=a_{ij}$.
For $t\in \{0,\ldots,m\}$ define $f_t:\{1,\ldots,n\}\to X$ by
$f_t(i)=\sum_{j=1}^n(A^{m-t})_{ij}f(j)$. A simple computation shows that if we set
$M_t^{(\ell)}=f_t\bigl(Z_t^{(\ell)}\bigr)$ then
$\bigl\{M_t^{(\ell)}\bigr\}_{t=0}^m$ is a martingale with respect
to the filtration induced by $\bigl\{Z_t^{(\ell)}\bigr\}_{t=0}^m$. We can therefore apply~\eqref{eq:pisier variant} to this martingale, and then average over  $\ell\in
\{1,\ldots,n\}$, to get the inequality:
\begin{equation}\label{eq:in coordinates-intro}
\sum_{t=1}^m
\sum_{i=1}^n\sum_{j=1}^n
a_{ij}\left\|f_t(i)-f_{t-1}(j)\right\|_X^2
\ifsodaelse{\\}{}
\lesssim m^{1-\frac{2}{p}}\sum_{i=1}^n\sum_{j=1}^n
(A^m)_{ij}\left\|x_i-\sum_{r=1}^n(A^m)_{jr}x_r\right\|_X^2.
\end{equation}
If we choose $y_i=\frac{1}{m}\sum_{s=0}^{m-1}
(A^s)_{ij}x_j$ for $i\in \{1\ldots,n\}$,
then we can use convexity to deduce the bound:
\begin{align}
\nonumber\sum_{t=1}^m
\ifsodaelse{&}{}
\sum_{i=1}^n\sum_{j=1}^n
\ifsodaelse{}{&}
a_{ij}\left\|f_t(i)-f_{t-1}(j)\right\|_X^2
\ge m \nonumber
\sum_{i=1}^n\sum_{j=1}^n
a_{ij}\left\|\frac{1}{m}\sum_{t=1}^m\left(f_t(i)-f_{t-1}(j)\right)\right\|_X^2  \displaybreak[2]  \\ \nonumber
&= m \sum_{i=1}^n\sum_{j=1}^n
a_{ij}\left\|y_i-y_j-\frac{1}{m}\sum_{r=1}^n(A^m)_{jr}(x_j-x_r)\right\|_X^2\\
\label{eq:get factor m-intro}
&\ge \frac{m}{2}\sum_{i=1}^n\sum_{j=1}^n
a_{ij}\left\|y_i-y_j\right\|_X^2
\ifsodaelse{\\ \nonumber & \qquad\qquad}{}
-\frac{1}{m}\sum_{j=1}^n\sum_{r=1}^n(A^m)_{jr}
\left\|x_j-x_r\right\|_X^2.
\end{align}
At the same time we can bound the right-hand side of~\eqref{eq:in
coordinates-intro} as follows:
\begin{align}
\nonumber
\sum_{i=1}^n\sum_{j=1}^n &
(A^m)_{ij}\left\|x_i-\sum_{r=1}^n(A^m)_{jr}x_r\right\|_X^2
 \nonumber \le
\sum_{i=1}^n\sum_{j=1}^n\sum_{r=1}^n
(A^m)_{ij}(A^m)_{jr}\left\|x_i-x_r\right\|_X^2
\\ \nonumber  & \le
2\sum_{i=1}^n\sum_{j=1}^n\sum_{r=1}^n
\ifsodaelse{\\ \nonumber & \qquad}{}
(A^m)_{ij}(A^m)_{jr}\left(\|x_i-x_j\|_X^2+\|x_j-x_r\|_X^2\right)
\\ &= 4\sum_{i=1}^n\sum_{j=1}^n(A^m)_{ij}\|x_i-x_j\|_X^2. \label{eq:LHS factor m-intro}
\end{align}
We note that:
\begin{align*}
\sum_{i=1}^n\sum_{j=1}^n &(A^m)_{ij}\|x_i-x_j\|_X^2
 =
\sum_{i=1}^n\sum_{j=1}^n\left(\frac{1}{m}\sum_{t=0}^{m-1}A^tA^{m-t}\right)_{ij}\|x_i-x_j\|_X^2
\\&\le
\frac{2}{m}\sum_{i=1}^n\sum_{j=1}^n\sum_{r=1}^n\sum_{t=0}^{m-1}
\ifsodaelse{\\ & \qquad}{}
(A^t)_{ir}(A^{m-t})_{rj}
\left(\|x_i-x_r\|_X^2+\|x_r-x_j\|_X^2\right) \displaybreak[2] \\
&= 2 \sum_{i=1}^n\sum_{j=1}^n
\left(\frac{1}{m}\sum_{t=0}^{m-1}A^t\right)_{ij}\|x_i-x_j\|_X^2
\ifsodaelse{\\ & \qquad}{}
+2
\sum_{i=1}^n\sum_{j=1}^n
\left(\frac{1}{m}\sum_{t=1}^{m}A^t\right)_{ij}\|x_i-x_j\|_X^2\\
&\le 4 \sum_{i=1}^n\sum_{j=1}^n
\left(\frac{1}{m}\sum_{t=0}^{m-1}A^t\right)_{ij}\|x_i-x_j\|_X^2
\ifsodaelse{\\ & \qquad}{}
+\frac{2}{m}\sum_{i=1}^n\sum_{j=1}^n(A^m)_{ij}\|x_i-x_j\|_X^2,
\end{align*}
which implies (by separating the cases $m\ge 4$ and $m< 4$) the bound:
\begin{equation}\label{eq:bound by average-intro}
\sum_{i=1}^n\sum_{j=1}^n(A^m)_{ij}\|x_i-x_j\|_X^2
\ifsodaelse{\\}{}
\lesssim\sum_{i=1}^n\sum_{j=1}^n\left(\frac{1}{m}\sum_{t=0}^{m-1}A^t\right)_{ij}\|x_i-x_j\|_X^2.
\end{equation}

 Substituting~\eqref{eq:get factor m-intro}
and~\eqref{eq:LHS factor m-intro} into~\eqref{eq:in coordinates-intro} yields
the bound:
\begin{multline}\label{one of the terms in cotype-intro}
m^{2/p}\sum_{i=1}^n\sum_{j=1}^n a_{ij}\left\|y_i-y_j\right\|_X^2\lesssim \sum_{i=1}^n\sum_{j=1}^n(A^m)_{ij}\|x_i-x_j\|_X^2
\\ \stackrel{\eqref{eq:bound
by average-intro}}{\lesssim}
\sum_{i=1}^n\sum_{j=1}^n\left(\frac{1}{m}\sum_{t=0}^{m-1}A^t\right)_{ij}\|x_i-x_j\|_X^2.
\end{multline}
At the same time,
\begin{multline}\label{eq:displacement-intro}
\sum_{i=1}^n
\|x_i-y_i\|_X^q=\sum_{i=1}^n\sum_{j=1}^n\left\|\frac{1}{m}\sum_{t=0}^{m-1}(A^t)_{ij}(x_i-x_j)\right\|_X^q
\\ \le
\sum_{i=1}^n\sum_{j=1}^n\left(\frac{1}{m}\sum_{t=0}^{m-1}A^t\right)_{ij}\|x_i-x_j\|_X^q.
\end{multline}
Inequalities~\eqref{one of the terms in cotype-intro}
and~\eqref{eq:displacement-intro} imply~\eqref{eq:intro def cotype} with $\e=\frac{2}{p}$, as required.

\subsection{Constructing the base graph: the heat semigroup on the tail space}\label{sec:base}

In this section we outline the proof of Lemma~\ref{lem:base-intro}
--- the existence of ``sufficiently good" base graphs for the class of $K$-convex
spaces (which contain the class of super-reflexive spaces as a
subclass). Some of the details are deferred to
Section~\ref{sec:base-details}.

We begin with few definitions. Let $\F_2=\{0,1\}$ be the field of two elements, and $(X,\|\cdot\|_X)$ a given normed space. $L_p(X)$ denotes the normed space of functions $f:V\to X$, where $V$ is an (implicit) finite set (in this paper, $V$ is the vertex set of some graph, and in this section it is $V=\F_2^n$) and the norm is \(\|f\|_{L_p(X)}= \bigl(\frac1{|V|} \sum_{x\in V} \|f(x)\|_X^p\bigr )^{1/p}. \)
Given a $V\times V$ symmetric stochastic matrix $A$, we view $A$ as a linear operator over $L_p(X)$: For $f\in L_p(X)$ we define $Af$ as
$(Af)(i)=\sum_j A_{ij}f(j)$ --- essentially we identify $A$ with $A\otimes I_X$.
For $f\in L_2(X)$,
let $\widehat f(A)=\E \left[f(x)W_A(x)\right]$ be the Fourier coefficient corresponding to the Walsh function $W_A(x)=(-1)^{\sum_{i\in A}x_i}$, where the expectation is with respect to the uniform probability measure on $\F_2^n$. Let $\partial_if(x)=\frac{f(x+e_i)-f(x)}{2}$ (where $\{e_i\}_{i=1}^n$ is the standard basis) and $\Delta f=\sum_{i=1}^n\partial_if$.

The starting point of our construction is the example of Khot and
Naor~\cite{KN06} of quotients of the hypercube by good codes, as
examples of metric spaces for which Bourgain's embedding
theorem~\cite{Bourgain-embed} is asymptotically sharp. It would be
instructive to first recall the argument from~\cite{KN06}. Thinking
of the cube $\F_2^n$ as an $n$-dimensional vector space over $\F_2$,
let $C\subseteq \F_2^n$ be a linear subspace of dimension at least
$\frac{n}{10}$ such that the minimum number of $1$'s in any non-zero
element of $C$ is $m\ge \frac{n}{10}$ (a good code). Given a
function $f:\F_2^n/C^\perp\to X$ we will think of $f$ as a function
defined on all of $\F_2^n$ which is constant on cosets of $C^\perp$.
The following claim is a simple observation of~\cite{KN06}:
\begin{claim} \label{cl:KN-cosets}
Let $C\subseteq \mathbb F_2^n$ be a linear code with minimum
distance of $m$, and fix $f:\mathbb F_2^n \to X$ which is constant
on cosets of $C^\perp$. Then $\widehat f(A)=0$ for all nonempty
$A\subseteq \{1,\ldots,n\}$ of cardinality less than $m$.
\end{claim}

 Since $\Delta W_A=|A|W_A$ and the (nonempty) Fourier spectrum of $f$ is supported on sets of size at least $m$, it follows from Parseval's identity that if $X$ is a Hilbert space then $\|\Delta f\|_{L_2(X)}\ge m\|f-\E f\|_{L_2(X)}$. But $\|\Delta f\|_{L_2(X)}=\left\|\sum_{i=1}^n\partial_i f\right\|_{L_2(X)}\le \sum_{i=1}^n\|\partial_i f\|_{L_2(X)}$ and $\|f-\E f\|_{L_2(X)}^2=\frac12 \E_{x,y}\left[\|f(x)-f(y)\|_2^2\right]$. An application of Cauchy-Schwarz and the fact that $m\asymp n$ now implies that $\E_{x,y}\left[\|f(x)-f(y)\|_2^2\right]\lesssim \frac{1}{n}\sum_{i=1}^n\|\partial_i f\|_{L_2(X)}^2$. A moment of thought reveals that this is the desired Poincar\'e inequality (in the case of Hilbert space), where the graph in question is on the vertex set $\F_2^n/C^\perp$, and each coset $x+C^\perp$ is joined by an edge to the cosets $\{x+e_i+C^\perp\}_{i=1}^n$. In fact, this graph has degree logarithmic in the number of vertices, which is much better than the assertion in Lemma~\ref{lem:base-intro}.

In order to make this idea work for a non-Hilbertian normed space $X$ it would be desirable to prove that if $f:\F_2^n\to X$ is in the $m$-tail space, i.e., the subspace of $L_2(X)$, denoted $L_2^{\ge m}(X)$, consisting of functions $f:\F_2^n\to X$ with $\widehat f(A)=0$ whenever $|A|<m$, then $\|\Delta f\|_{L_2(X)}\gtrsim m\|f\|_{L_2(X)}$. This fact is false without some additional assumption on the geometry of $X$ --- in the full version of this extended abstract we observe that it fails when $X=L_1$.
Recall from Section~\ref{sec:coarse} that a Banach space $X$ is called $K$-convex if there exists $\e_0>0$ and $n_0\in \N$ such that any embedding of $\ell_1^{n_0}$ into $X$ incurs distortion at least $1+\e_0$.
Here we will show that if $X$ is $K$-convex then there exists $\delta=\delta(X)>0$ such that if $f\in L_2^{\ge m}(X)$ then
\begin{equation} \label{eq:laplcian-K-convex}
\|\Delta f\|_{L_2(X)}\gtrsim m^\delta\|f\|_{L_2(X)}.
\end{equation}

 While this bound is insufficient for our purpose, from the the proof of~\eqref{eq:laplcian-K-convex} we can extract a graph which is more complicated than $\F_2^n/C^\perp$ (but is of independent interest), and whose degree bound to what is claimed in Lemma~\ref{lem:base-intro}, which will serve as the base graph(s) for $K$-convex spaces.

The key idea is to consider the heat semigroup $\{T_t\}_{t>0}$ of operators given by
\[ (T_t f)(x) = \sum_{A\subseteq [n]} e^{-t|A|} \hat f(A) W_A(x) .\]
and make use of the following lemma, which is the heart of our argument.

\begin{lemma} \label{lem:bounding-beckner}
If $X$ is a $K$-convex Banach Space, and $p>1$, then there exist $\alpha\ge 1$, $a>0$,
such that for every $n,m\in \mathbb N$, $m \le n$, every $t\ge 0$, and every $f:\mathbb F_2^n \to X$, $f\in L_p^{\ge m}(X)$,
\begin{equation}\label{eq:bounding-beckner-b}
 \|T_t f\|_{L_p(X)} \lesssim  e^{-am \min\{t,t^{\alpha}\}} \| f\|_{L_p(X)}  ,
 \end{equation}
\end{lemma}

The proof of Lemma~\ref{lem:bounding-beckner} is based on deep
results of Pisier on holomorphic extensions of the heat semigroup on
$K$-convex spaces~\cite{Pisier-K-convex}, and a quantitative version
of the proof of a recent factorization theorem of
Pisier~\cite{pisier-2007}. The proof is contained in
Section~\ref{sec:proof-of-bounding-beckner}.

Note that when $m\ge 1$, $\Delta$ is invertible on $L_p^{\ge m}(X)$,
and moreover we have the identity $\Delta^{-1}=\int_0^\infty T_t
dt$, since $\Delta^{-1}W_A=|A|^{-1} W_A$, and
\begin{equation*}
 \left(\int_0^\infty T_tdt \right ) W_A= \int_0^\infty T_t W_A dt
\ifsodaelse{\\}{}
= \int_0^\infty e^{-t|A|} W_A d_t= |A|^{-1} W_A.
\end{equation*}
Hence by integrating~\eqref{eq:bounding-beckner-b} over $t$, it immediately implies~\eqref{eq:laplcian-K-convex}.

Inequality~\eqref{eq:bounding-beckner-b} will serve as the basis for
our base graph. This is made transparent by the observation that the
heat semigroup can be viewed as a noise operator with noise rate
$\frac{1-e^{-t}}{2}$.  The following standard claim
\ifsodaelse{}{(whose proof is recalled in Section~\ref{sec:base-details})} formalizes this
statement.
\begin{claim} \label{cl:real-beckner}
For every $f:\mathbb F_2^n \to X$,  $x\in\mathbb F_2^n$, and $t\in (0,\infty)$,
\begin{equation} \label{eq:real-beckner-b}
(T_t f)(x) = \sum_{y\in\mathbb F_2^n}  \left( \frac{1-e^{-t}}2 \right)^{\|x-y\|_1}
\ifsodaelse{\\}{}
\cdot \left( \frac{1+e^{-t}}2 \right)^{n-\|x-y\|_1}
f(y).
\end{equation}
\end{claim}

Hence $T_t$ induces a natural weighted graph structure on $\F_2^n$: the weight of the edge joining $x,y\in \F_2^n$ is $\left(\frac{1-e^{-t}}{2}\right)^{\|x-y\|_1}\left(\frac{1+e^{-t}}{2}\right)^{n-\|x-y\|_1}$. Our base graph is ``morally" the quotient of this weighted cube by $C^\perp$. When $t\approx n^{-1/\alpha}$, the factor in the right hand side of~\eqref{eq:bounding-beckner-b} becomes a constant, and this implies a Poincar\'e inequality because of the following general estimate (proved in Section~\ref{sec:base-details}).

\begin{proposition} \label{prop:linear=>nonlin}
Fix a Banach space $X$, a symmetric stochastic $N\times N$ matrix
$A$, and $p\ge 1$. Assume that there exists $1>\lambda >0$ such that
for every $f:[N]\to X$ with $\sum_{i=1}^N f(i)=0$ we have,
\begin{equation} \label{eq:norm-bound}
\|Af\|_{L_p(X)} \le \lambda \|f\|_{L_p(X)} .\end{equation}
Then,
\[ \bpconst(A,\|\cdot\|_X^p) \le 8^p (1-\lambda)^{-p} . \]
\end{proposition}


Still, ``$T_t/C^\bot$" is not a low degree unweighted graph. To make
it an unweighted graph we first truncate the weights of $T_t$ below
an appropriate threshold, and approximate the remaining (large
enough) weights using multiple parallel edges. This process is
summarized in the following lemma (whose proof is deferred to the
full version due to lack of space).
\begin{lemma} \label{lem:noise-truncation-b}
Let $\tau=(1-e^{-t})/2$.
There exists a Cayley graph on $\mathbb F_2^n$, $G=(\mathbb F_2^n,E)$, having degree at most $\tau^{-4\tau n} (1-\tau)^{-(1-4\tau )n}$ such that
for every metric space $(X,d_X)$, $p\ge 1$ and $f,g:\mathbb F_2^n \to X$,
\begin{multline} \label{eq:trunc-noise}
 \ifsodaelse{}{\frac{1}{3|E|} \sum_{(x,y)\in E}  d_X(f(x),g(y))^p \\  \le}
  \frac{1}{2^n} \sum_{x,y\in \mathbb F_2^n} (T_t)_{xy}d_X(f(x),g(y))^p
\ifsodaelse{\\}{}
\le \frac{3}{|E|} \sum_{(x,y)\in E}  d_X(f(x),g(y))^p ,
\end{multline}
as long as $18 \tau^2 n \ge 2 p \log n +\log 4$.
\end{lemma}

Only then we pass to the quotient $\F_2^n/C^\perp$ by identifying
vertices within each coset and dividing the number of the resulting
``quotient edges" between each coset by $|C^\perp|$. This is
formally summarized as the proof of Lemma~\ref{lem:base-intro}.

\begin{proof}[Proof of Lemma~\ref{lem:base-intro}]
The argument works for any power $p\in(1,\infty)$ and not just $p=2$.
Fix a $K$-convex Banach space $X$, and let $\alpha\ge 1$ and $a>0$ as in Lemma~\ref{lem:bounding-beckner}.

Fix $n\in \mathbb N$.
Set $t= \min\{n^{-1/\alpha}/a, 0.1\}$, $\tau=(1-e^{-t})/2$ and apply the discretization process --- Lemma~\ref{lem:noise-truncation-b} --- on $T_t$, from which we obtain a Cayley graph $G=G_n=(\mathbb F_2^n,E)$ whose degree is at most
\begin{equation*}
 \tau^{-4\tau n} (1-\tau)^{(1-4\tau)n} \le \tau^{-8\tau n}
 \ifsodaelse{\\}{}
 \le \max \left\{40^{320\cdot10^{\10\alpha}}, n^{2n^{1-1/a\alpha}/\alpha}/a\right\} \lesssim e^{n^{1-\delta}} ,
 \end{equation*}
for some $\delta=\delta(X)\in(0,1)$, depending only on $\alpha$ and $a$ (and hence on $X$), but not on $n$.

Let $C\subseteq \F_2^n$ be a ``good" linear code of dimension at
least $n/10$ and minimum weight $m \ge n/10$. We define
$H=H_n=(\F_2^n/C^\perp,E/C^\perp)$: the quotient of $G$ by
$C^\perp$, whose vertex-set is the set of cosets of $C^\perp$ and in
the edge (multi)set $E/C^\perp$  the number of edges  between  two
cosets $x+C^\perp$ and $y+C^\perp$,  is the number of edges in $G$
with  one endpoint in $x+C^\perp$ and the other in $y+C^\perp$
\emph{divided by} $|C^\perp|$. Since $C^\perp$ is a linear subspace
in a vector space, the degree of $H$ is the same as the degree of
$G$.

Next, fix $f:\F_2^n/C^\perp \to X$ with $\sum_{x\in
\F_2^n/C^\perp}f(x)=0$. Thinking on $f$ as a function on $\F_2^n$
which is constant on cosets of $C^\perp$, we have by
Claim~\ref{cl:KN-cosets} that $\widehat f(A)=0$ for every
$\emptyset\neq A\subseteq\{1,\ldots,n\}$, $|A|< m$. Hence by
Lemma~\ref{lem:bounding-beckner},
\[ \|T_t f\|_{L_p(X)}\lesssim e^{- am t^\alpha} \|f\|_{L_p(X)} \le \lambda \|f\|_{L_p(X)} ,\]
where $0<\lambda<1$ is a universal constant. Since this is true for
any $f:\F_2^n/C^\perp \to X$ with $\sum_{x\in \F_2^n/C^\perp}
f(x)=0$, applying Proposition~\ref{prop:linear=>nonlin} we have that
for every $f,g:\F_2^n/C^\perp \to X$,
\begin{align}
\nonumber \frac{|C^\perp|^2}{2^{2n}} & \sum_{x,y\in\F_2^n/C^\perp} \|f(x)-g(y)\|_X^2
\\
\label{eq:base-intro-1} &\lesssim \frac{|C^\perp|}{2^n}
\ifsodaelse{\!\!\!}{}
\sum_{x,y\in \F_2^n/C^\perp}  \frac{1}{|C^\perp|}
\ifsodaelse{\\ \nonumber  &\quad}{}
\sum_{a,b\in C^\perp} (T_t)_{x+a,y+b} \|f(x+a)-g(y+b)\|_X^2 \displaybreak[1]\\
\nonumber & = \frac{1}{2^n} \sum_{x,y\in \F_2^n}  (T_t)_{x,y} \|f(x)-g(y)\|_X^2 \\
\label{eq:base-intro-3} & \le \frac{3}{|E|} \sum_{(x,y)\in E}\|f(x)-g(y)\|_X^2
\\
\nonumber & = \frac{3}{|E/C^\perp|} \sum_{(x,y)\in E/C^\perp}\|f(x)-g(y)\|_X^2.
\end{align}
The estimate~\eqref{eq:base-intro-1} follows from Proposition~\ref{prop:linear=>nonlin},
and~\eqref{eq:base-intro-3}
follows from~\eqref{eq:trunc-noise}, and is valid since  $18\tau^2 n\ge 4 \log n+\log 4$.
\end{proof}

\subsection{Discussion}

In the full version of this extended abstract we will also establish the metric Markov cotype
property~\eqref{eq:intro def cotype} for other classes of metric spaces, including $CAT(0)$
metric spaces, and in particular simply connected manifolds with non-positive sectional
curvature (in which case we can take $C=O(1)$ and $\e=1$). We will discuss this notion
in detail and show that it implies a slight variant of Ball's Markov cotype. This
fact is new for $CAT(0)$ metric spaces, and hence, in conjunction with Ball's
extension theorem~\cite{Ball}, our work implies a new Lipschitz extension theorem for
$CAT(0)$ targets: any Lipschitz function from a subset $U$ of a metric space
$(X,d_X)$ which has Markov type $2$ and takes values in a $CAT(0)$ metric space $(Y,d_Y)$ can be extended
to a Lipschitz function defined on all of $X$ whose Lipschitz constant is larger by
at most a constant factor (depending on $X$). The definition of the notion of Markov
type is beyond the scope of this extended abstract: it suffices to say that we can
take $X$ to be Hilbert space\ifsodaelse{}{~\cite{Ball}}, or $L_p$ for $p\in [2,\infty)$ and even
more generally a normed spaces whose modulus of smoothness has power type
$2$\ifsodaelse{}{~\cite{NPSS06}}, a tree, the word metric on a hyperbolic group or a
simply connected manifold with pinched sectional curvature\ifsodaelse{}{~\cite{NPSS06}},
series parallel graphs\ifsodaelse{}{~\cite{BKL07}}, or Alexandrov spaces
(in particular manifolds) of non-negative curvature\ifsodaelse{}{~\cite{Ohta08}}.
While we believe that this extension theorem is a key consequence
of our work, for lack of space we defer the discussion about it to the full version of this paper.

Another major addition to the full version of this extended abstract
will be a different proof of the decay of the Poincar\'e constant of
powers of symmetric stochastic matrices in super-reflexive spaces
and in $CAT(0)$ spaces. This proof is based on using \emph{norm
bounds}. The norm bound $\lambda(A,X,p)$ for a doubly stochastic
$n\times n$ matrix $A$ in a normed space $X$ is the smallest
$\lambda>0$ satisfying~\eqref{eq:norm-bound}.
The quantity $\lambda(A,\mathbb{R},2)$ is the second absolute
eigenvalue of $A$, and Proposition~\ref{prop:linear=>nonlin} shows
that it always controls $\bpconst(A,\|\cdot\|_X^p)$ from above. We
show in the full version that when $X$ is $p$-convex,
$\lambda(A,X,p)$ also controls $\bpconst(A,\|\cdot\|_X^p)$ from
below, and this allows us to obtain a short proof of the decay of
$\bpconst(A^t,\|\cdot\|_X^p)$ as a function  of $t$ when $X$ is a
$p$-convex space. This alternative approach has the advantage that
it proves the decay of the Poincar\'e constant of the power of the
matrix instead of its Ces\'aro average, and it is also shorter than
the Markov cotype approach. However, in many respects it is less
``robust" than the Markov cotype approach taken here, it yields
worse bounds, and most importantly, it yields a Poincar\'e
inequality which is insufficient for our construction of zigzag
super-expanders. All of these issues will be discussed in the full
version.

\ifsodaelse{}{
A major question left open is what are the geometric conditions
which allows for ``sufficient" decay (for the purposes of our zigzag
iteration) of the Poincar\'e constant. We give an example in
 Section~\ref{sec:no-decay} for a metric space in which the decay is insufficient.
 However, there are natural spaces like $L_\infty$ and $L_1$ for which we do not know the answer.
 Indeed, it may be the case that any normed space has a sufficient decay.

Another question is what are the geometric conditions on metric
spaces for the existence of expanders with respect to those spaces.
It is tempting to conjecture that for normed spaces, having such an
expander is equivalent to having finite cotype (i.e., $\exists
\e_0>0, n_0\in \mathbb N$ such that $\ell_\infty^{n_0}$ does not
embed with distortion less than $1+\e_0$).

As was mentioned in Section~\ref{sec:coarse}, it is an open question
whether every ``classical" expander is also super-expander. In the
full version we rule out the most obvious approach to prove such a
result: coarsely embedding any expander family in any other expander
family. We give an example of two expander families $\mathcal F_1$
and $\mathcal F_2$  such that $\mathcal F_1$ does not coarsely embed
in $\mathcal F_2$.

It seems to be a challenging problem whether
in~\eqref{eq:laplcian-K-convex} we can take $\delta=1$. Even in the
case $X=L_p$, $p\ne 2$, this seems to be unknown. A classical result
of Meyer~\cite{Meyer} implies that for $p\ge 2$ we can take
$\delta=\frac12$, but even this improved bound is insufficient for
our purpose of constructing a ``sufficiently good" base graph.
}

\section{Some Details} \label{sec:some-proofs}

\subsection{Markov cotype and the decay of $\gamma$}
\label{sec:details-decay}

\begin{proof}[Proof of Claim~\ref{cl:decay for gamma}]
Write $B= (b_{ij})= \frac{1}{m}\sum_{t=0}^{m-1}A^t$. We may assume that $\gamma\left(B,d_X^2\right)> 12C^2$, since otherwise we are done. This assumption implies that there exist $x_1,\ldots,x_n\in X$ such that:
\begin{equation}\label{eq:reverse poin}
\frac{1}{n^2}\sum_{i=1}^n\sum_{j=1}^n d_X(x_i,x_j)^2>\frac{12C^2}{n}\sum_{i=1}^n\sum_{j=1}^n b_{ij} d_X(x_i,x_j)^2.
\end{equation}
Let $y_1,\ldots,y_n\in X$ be as in~\eqref{eq:intro def cotype}. Note that for all $i,j\in \{1,\ldots,n\}$ we have:
\begin{multline}\label{eq:triange square}
d_X(x_i,x_j)^2\le \left[d_X(x_i,y_i)+d_X(y_i,y_j)+d_X(y_j,x_j)\right]^2
\\ \le 3d_X(x_i,y_i)^2+3d_X(y_i,y_j)^2+3d_X(y_j,x_j)^2.
\end{multline}
By averaging~\eqref{eq:triange square} we get the bound:
\begin{align}
\nonumber\frac{1}{n^2} &\sum_{i=1}^n\sum_{j=1}^n d_X(y_i,y_j)^2
\nonumber \ge \frac{1}{3n^2}\sum_{i=1}^n\sum_{j=1}^n d_X(x_i,x_j)^2-\frac{2}{n}\sum_{i=1}^nd_X(x_i,y_i)^2 \displaybreak[1]\\
&=\nonumber \frac{1}{6n^2}\sum_{i=1}^n\sum_{j=1}^n d_X(x_i,x_j)^2
\ifsodaelse{\\ \nonumber & \qquad}{}
+\frac{1}{6n^2}\sum_{i=1}^n\sum_{j=1}^n d_X(x_i,x_j)^2-\frac{2}{n}\sum_{i=1}^nd_X(x_i,y_i)^2
\displaybreak[1]\\
&\stackrel{\eqref{eq:reverse poin}}{>} \nonumber \frac{1}{6n^2}\sum_{i=1}^n\sum_{j=1}^n d_X(x_i,x_j)^2
\ifsodaelse{\\ \nonumber & \qquad}{}
+\frac{2C^2}{n}\sum_{i=1}^n\sum_{j=1}^n b_{ij} d_X(x_i,x_j)^2-\frac{2}{n}\sum_{i=1}^nd_X(x_i,y_i)^2\\
&\stackrel{\eqref{eq:intro def cotype}}{\ge} \frac{1}{6n^2}\sum_{i=1}^n\sum_{j=1}^n d_X(x_i,x_j)^2.
\label{eq:left part of cotype}
\end{align}
Using~\eqref{eq:intro def cotype} once more, in conjunction with the definition of $\gamma\left(A,d_X^2\right)$, we see that:
\begin{multline*}
C^2\sum_{i=1}^n\sum_{j=1}^n b_{ij}d_X(x_i,x_j)^2
\ifsodaelse{\\}{}
\ge \frac{m^\e}{\gamma(A,d_X^2)}\cdot \frac{1}{n^2}\sum_{i=1}^n\sum_{j=1}^n d_X(y_i,y_j)^2 \\ \stackrel{\eqref{eq:left part of cotype}}{>} \frac{m^\e}{6\gamma(A,d_X^2)}\cdot \frac{1}{n^2}\sum_{i=1}^n\sum_{j=1}^n d_X(x_i,x_j)^2.
\end{multline*}
This concludes the proof of~\eqref{eq:decay for gamma}.
\end{proof}

\begin{proof}[Proof of Claim~\ref{cl:gamma+-as-tensor}]
Fix $f,g:[n]\to X$, and define $h:[2n]\to X$,
\[ h(x)= \begin{cases} f(x) \ & x\le n\\ g(x-n) &x>n. \end{cases} \]
Then
\begin{align*} \frac1{n^{2}}
\ifsodaelse{&}{} \sum_{x,y\in [n]} d_X(f(x),g(y))^2
\ifsodaelse{\\}{}
&= \frac1{n^{2}} \sum_{x,y\in [n]} d_X(h(x),h(y+n))^2
\\ & \le \frac{2}{ (2n)^{2}} \sum_{x,y\in [2n]} d_X(h(x),h(y))^2
\displaybreak[1]
\\ & \le    \gamma\left( \left( \begin{smallmatrix}   0 & A \\    A & 0 \end{smallmatrix} \right),d_X^2\right)
\frac{2}{2n} \sum_{x,y\in [n]} 2A_{xy} d_X(h(x),h(y+n))^2
\\ & = \frac{2 \gamma\left( \left( \begin{smallmatrix}    0 & A \\    A & 0 \end{smallmatrix} \right),d_X^2\right) }
{n} \sum_{x,y\in [n]} A_{xy} d_X(f(x),g(y))^2.
\end{align*}

We next prove the lower bound.
Fix $h:[2n]\to X$, and define $f,g:[n]\to X$, $f(x)=h(x)$, and $g(x)=h(x+n)$.
Then,
\begin{multline*}  \frac1{n^{2}}\sum_{x,y\in [n]} d_X(h(x),h(y))^2
 \le \frac1{n^2}\sum_{x,y\in [n]} \ifsodaelse{\!\!\!}{} 2 \sum_{z\in [n]}
\left( \begin{aligned}
\ifsodaelse{&}{} d_X(h(x),h(z+n))^2 \ifsodaelse{\\ &}{} + d_X(h(y),h(z+n))^2 \end{aligned}\right )
\\
=
\frac 4 {n^2} \sum_{x,y\in [n]}d_X(f(x),g(y))^2 .
\end{multline*}
Similarly,
\begin{eqnarray*} &&\!\!\!\!\!\!\!\!\!\!\!\!\!\!\!\!}{\!\!\!\!\!\!\!\!\! \frac1{n^2}\sum_{x,y\in [n]} d_X(h(x+n),h(y+n))^2
\\  &\le&
\ifsodaelse{}{
  \frac1{n^2}\sum_{x,y\in [n]} \ifsodaelse{\!\!\!}{} 2 \sum_{z\in [n]}\left(
  \begin{aligned}
  \ifsodaelse{&}{} d_X(h(x+n),h(z))^2 \ifsodaelse{\\ &}{} + d_X(h(y+n),h(z))^2
  \end{aligned}
  \right)
  \\
  &=&
}
\frac 4 {n^2} \sum_{x,y\in [n]}d_X(f(x),g(y))^2 .
\end{eqnarray*}

Hence, by summing the above two inequalities,
\begin{align*}
\frac{1}{(2n)^{2}} &\sum_{\substack{x,y\in [2n]}} d_X(h(x),h(y))^2
\ifsodaelse{}{
  \\ &=\frac{1}{(2n)^{2}} \sum_{\substack{x,y\in [n]}} d_X(h(x),h(y))^2
  \ifsodaelse{\\ & \quad}{}
  + \frac{1}{(2n)^{2}} \sum_{\substack{x,y\in [n]}} d_X(h(x+n),h(y+n))^2\\
  & \quad
  +\frac{1}{(2n)^{2}} \sum_{\substack{x,y\in [n]}} d_X(h(x),h(y+n))^2
  \ifsodaelse{\\ & \quad}{}
  + \frac{1}{(2n)^{2}} \sum_{\substack{x,y\in [n]}} d_X(h(x+n),h(y))^2 \displaybreak[1]
}
\\ & \le \frac{1+1+0.25+0.25}{n^{2}} \sum_{x,y\in [n]}d_X(f(x),g(y))^2 \displaybreak[1]
\\ & \le \frac{2.5 \bpconst(G,d_X^2)}{n} \sum_{x,y\in [n]} A_{xy} d_X(f(x),g(y))^2
\\ &  =  \frac{2.5 \bpconst(G,d_X^2)}{n} \ifsodaelse{\!\!}{} \sum_{x,y\in [n]}
\ifsodaelse{\!\!}{} A_{xy} d_X(h(x),h(y+n))^2
. \qedhere
\end{align*}
\end{proof}

\begin{proof}[Proof of Claim~\ref{cl:commute tenzor}]
We observe that for odd $t$,
$\bigl(\begin{smallmatrix} 0 & A\\ A& 0 \end{smallmatrix}\bigr)^t= \bigl (\begin{smallmatrix} 0 & A^t\\ A^t& 0 \end{smallmatrix}\bigr)$.
What complicates matters is that for even $t$,
$\bigl(\begin{smallmatrix} 0 & A\\ A& 0 \end{smallmatrix}\bigr)^t= \bigl (\begin{smallmatrix}   A^t & 0\\ 0 & A^t \end{smallmatrix}\bigr)$.
However, we can write every even $t>0$ as a sum of two almost unique odd numbers
by defining for even $t$, $o_1(t)\le o_2(t)$ the unique pair of odd numbers
such that $o_2(t)-o_1(t)\le 2$, and $o_1(t)+o_2(t)=t$.
Note that the multiset
\[ \{o_1(t),o_2(t):\ t\text{ is even, and } t\le m \}\]
contains all the odd numbers in the range $\{1,\ldots \lfloor m/2 \rfloor+1\}$,
and at most four items for each value. Let $B=\bigl ( \begin{smallmatrix} 0 & A \\ A& 0 \end{smallmatrix}\bigr)$.
We therefore can bound for every $h:[2n]\to X$,
\begin{align*}
 \frac1{2n} &\sum_{x,y\in [2n]}  (\A_m(B))_{xy} d_X(h(x),h(y))^2
\ifsodaelse{}{
   =
  \frac{1}{2nm} \sum_{t=0}^{m-1} \sum_{x,y \in [2n]} (B^t)_{xy} d_X(h(x),h(y))^2
}
\displaybreak[0]
\\  & \le
\frac{1}{2nm} \sum_{\substack{t\in[m]\\ t\text { is odd}}}
\sum_{x,y \in [2n]} (B^t)_{xy} d_X(h(x),h(y))^2
\ifsodaelse{\\ & \quad}{}
 +\frac{1}{2nm} \sum_{x\in[2n]} d_X(h(x),h(x))^2
\\  &  +
\frac{1}{2nm} \sum_{\substack{t\in\{2,\ldots,m-1\}\\ t\text { is even}} }
\ifsodaelse{\\ & \qquad\qquad}{}
2\biggl( \sum_{x,y\in[2n] } (B^{o_1(t)})_{xy} d_X(h(x),h(y))^2
 +\sum_{x,y \in [2n]} (B^{o_2(t)})_{xy}d_X(h(x),h(y))^2 \biggr )
 \displaybreak[1] \\   & \le \frac{9}{2nm}
\sum_{\substack{t\in[m]\\ t\text { is odd}}}
\sum_{x,y\in[2n] }  (B^t)_{xy}  d_X(h(x),h(y))^2
\ifsodaelse{}{
  \\  & = \frac{9}{2nm}
  \sum_{\substack{t\in[m]\\ t\text { is odd}}}
  \sum_{x,y\in[2n] }  \bigl(\begin{smallmatrix} 0 & A^t\\ A^t & 0 \end{smallmatrix} \bigr)_{xy}
  d_X(h(x),h(y))^2
}
\\  & \le
\frac{9}{2n}  \sum_{x,y\in[2n]}  \bigl(\begin{smallmatrix} 0 & \sarkozy_m(A)\\ \sarkozy_m(A) &0 \end{smallmatrix}\bigr)d_X(h(x),h(y))^2  .
\end{align*}
Hence
\begin{multline*}
 \frac{1}{(2n)^2} \sum_{x,y\in[2n]} d_X(h(x),h(y))^2
\le
 \frac{\pconst\left(\sarkozy_m\bigl(\begin{smallmatrix} 0 & A\\ A & 0 \end{smallmatrix} \bigr),d_X^2\right)}
 {2n}
\cdot
 \sum_{x,y\in[2n]} \left( \A_m \bigl(\begin{smallmatrix} 0 & A\\ A & 0 \end{smallmatrix} \bigr)\right)_{xy} d_X(h(x),h(y))^2
 \ifsodaelse{\displaybreak[1]}{}
\\  \le \frac{ 9 \pconst\left(\sarkozy_m\bigl(\begin{smallmatrix} 0 & A\\ A & 0 \end{smallmatrix} \bigr),d_X^2\right)} {2n}
\cdot \sum_{x,y\in[2n]}  \bigl(\begin{smallmatrix} 0 & \sarkozy_m(A)\\ \sarkozy_m(A) &0 \end{smallmatrix}\bigr)d_X(h(x),h(y))^2,
\end{multline*}
which means that
\begin{equation*}
\pconst\left (\bigl(\begin{smallmatrix} 0 & \sarkozy_m(A)\\
\sarkozy_m(A) &0 \end{smallmatrix}\bigr), d_X^2\right ) \le
9\pconst\left(\sarkozy_m\bigl(\begin{smallmatrix} 0 & A\\ A & 0
\end{smallmatrix} \bigr),d_X^2\right). \qedhere
\end{equation*}
\end{proof}

\subsection{Proof of~\eqref{eq:pisier variant}}
\label{sec:proof-pisier-variant}

Our methods here are inspired by Ball's proof~\cite{Ball} of the
Markov cotype property.
Let $(X,\|\cdot\|_X)$ be a normed space. The {\em modulus of uniform
convexity\/} of $X$ is defined for $\e\in [0,2]$ as
\begin{equation}\label{def:convexity}
\delta_X(\e)=\inf\left\{1-\frac{\|x+y\|_X}{2}: \begin{aligned}  & \ x,y\in X,
\ifsodaelse{\\ &}{\ }
\|x\|_X=\|y\|_X=1,
\ifsodaelse{\\ &}{\ }
\|x-y\|_X=\e
\end{aligned}
\right\}.
\end{equation}
A normed space $X$ is said to be {\bf uniformly convex} if
$\delta_X(\e)>0$ for all $\e\in (0,2]$. Furthermore, $X$ is said to
have modulus of convexity of power type $p$ if there exists a
constant $c$ such that $\delta(\e)\ge c\,\e^p$ for all $\e\in
[0,2]$. It is straightforward to check that in this case necessarily
$p\ge 2$. By Proposition 7 in~\cite{BCL} (see also~\cite{Fiegel76}),
$X$ has modulus of convexity of power type $p$ if and only if there
exists a constant $K>0$ such that for every $x,y\in X$
\begin{equation}\label{eq:two point convex}
2\,\|x\|_X^p+\frac{2}{K^p}\,\|y\|_X^p\le \|x+y\|_X^p+\|x-y\|_X^p.
\end{equation}
The least $K$ for which \eqref{eq:two point convex} holds is called
the $p$-convexity constant of $X$, and is denoted $K_p(X)$.

The following lemma is stated and proved in~\cite{Ball} when
$p=2$.

\begin{lemma}\label{lem:from Ball}
Let $X$ be a normed space and $U$ a random vector in $X$ with $\E
\|U\|_X^p<\infty$. Then
\begin{equation} \label{eq:p-convex-rv}
\|\E U\|_X^p+\frac{\E\|U-\E U\|_X^p}{(2^{p-1}-1)K_p(X)^p}\le
\E\|U\|_X^p.
\end{equation}
\end{lemma}

\begin{proof}
We repeat here the $p>2$ variant of the argument from~\cite{Ball}
for the sake of completeness. Define
\begin{equation}\label{eq:def theta}
\theta=\inf\biggl\{\frac{\E \|V\|_X^p-\|\E V\|_X^p}{\E\|V-\E
V\|_X^p}:
\ifsodaelse{\\}{\ }
V\in L_p(X)\ \wedge\ \E\|V-\E V\|_X^p>0\biggr\}.
\end{equation}
Then $\theta\ge 0$. Our goal is to show that
\begin{equation}\label{eq:goal lower theta}
\theta\ge \frac{1}{(2^{p-1}-1)K_p(X)^p}.
\end{equation}
Fix $\phi>\theta$. Then there exists a random vector $V_0\in L_p(X)$
for which
\begin{equation}\label{eq:reverse theta}
\phi \E\|V_0-\E V_0\|_X^p> \E \|V_0\|_X^p-\|\E V_0\|_X^p.
\end{equation}
Fix $K>K_p(X)$. Apply the inequality~\eqref{eq:two point convex}
point-wise to the vectors $x=\frac12 V_0+\frac12 \E V_0$ and
$y=\frac12 V_0-\frac12 \E V_0$, to get that
\begin{equation}\label{eq:point-wise}
2\left\|\frac12 V_0+\frac12 \E
V_0\right\|_X^p+\frac{2}{K^p}\left\|\frac12 V_0-\frac12 \E
V_0\right\|_X^p
\ifsodaelse{\\}{}
\le \|V_0\|_X^p+\|\E V_0\|_X^p.
\end{equation}
Hence
\begin{align*}
\phi \ifsodaelse{&}{} \E \ifsodaelse{}{&} \|V_0-\E V_0\|_X^p
\stackrel{\eqref{eq:reverse theta}}{>} \E
\|V_0\|_X^p-\|\E V_0\|_X^p
\displaybreak[1]\\
&\stackrel{\eqref{eq:point-wise}}{\ge} \ifsodaelse{\!\!}{} 2\left(\E\left\|\frac12
V_0+\frac12 \E V_0\right\|_X^p
\ifsodaelse{\!\!}{}
-\left\|\E \left(\frac12V_0+\frac12\E
V_0\right)\right\|_X^p\right)
  + \frac{2}{K^p}\E\left\|\frac12
V_0-\frac12 \E V_0\right\|_X^p
\displaybreak[1] \\
&\stackrel{\eqref{eq:def
theta}}{\ge} 2\E\left\|\left(\frac12V_0+\frac12\E V_0\right)-\E
\left(\frac12V_0+\frac12\E
V_0\right)\right\|_X^p
\ifsodaelse{\\ & \qquad}{}
+\frac{2}{K^p}\E\left\|\frac12 V_0-\frac12 \E
V_0\right\|_X^p\\
&=
\left(\frac{\theta}{2^{p-1}}+\frac{1}{2^{p-1}K^p}\right)\E\|V_0-\E
V_0\|_X^p.
\end{align*}
Thus
\begin{equation}\label{eq:phi}
\phi\ge \frac{\theta}{2^{p-1}}+\frac{1}{2^{p-1}K^p}.
\end{equation}
Since~\eqref{eq:phi} holds for all $\phi>\theta$ and $K>K_p(X)$, the
desired lower bound~\eqref{eq:goal lower theta} follows.
\end{proof}

Let $X$ be a Banach space with $K_p(X)<\infty$. Assume that
$\{M_k\}_{k=0}^n\subseteq X$ is a martingale with respect to the
filtration $\mathcal{F}_0\subseteq \mathcal{F}_1\subseteq
\cdots\subseteq \mathcal{F}_{n-1}$; that is, $\E\left(M_{i+1}
|\mathcal{F}_{i}\right)=M_i$ for every $i\in\{0,1,\dots,n-1\}$.
Lemma~\ref{lem:from Ball} implies that
\begin{align}
\E\bigl[&\left.\left\|M_n-M_0\right\|_X^p\right|\mathcal{F}_{n-1}\bigr]\nonumber
\\ \nonumber &\ge
\left\|\E\left[\left.
M_n-M_0\right|\mathcal{F}_{n-1}\right]\right\|_X^p
 +\frac{
\E\left[\left.\left\|M_n-M_0-\E\left[\left.M_n-M_0\right|\mathcal{F}_{n-1}\right]\right\|_X^p\right|\mathcal{F}_{n-1}\right]}{(2^{p-1}-1)K_p(X)^p}\nonumber
\\  &= \left\|M_{n-1}-M_0\right\|_X^p
+\frac{\E\left[\left.\left\|M_n-M_{n-1}\right\|_X^p\right|\mathcal{F}_{n-1}\right]}{(2^{p-1}-1)K_p(X)^p}.
\label{eq:pisier iteration}
\end{align}
Taking expectation in~\eqref{eq:pisier iteration} implies that
\begin{equation*}
\E\left[\left\|M_n-M_0\right\|_X^p\right]
\ifsodaelse{\\}{}
\ge
\E\left[\left\|M_{n-1}-M_{0}\right\|_X^p\right]+\frac{\E\left[\left\|M_n-M_{n-1}\right\|_X^p\right]}
{(2^{p-1}-1)K_p(X)^p}.
\end{equation*}
Iterating this argument we obtain the following famous inequality of
Pisier~\cite{Pisier-martingales}.

\begin{theorem}[Pisier~\cite{Pisier-martingales}]\label{thm:pisier ineq} Let $X$ be a
Banach space with $K_p(X)<\infty$. Assume that
$\{M_k\}_{k=0}^n\subseteq X$ is a martingale (with respect some
filtration). Then
\[
\E \left[\left\|M_n-M_0\right\|_X^p\right]\ge
\frac{\sum_{k=1}^n\E\left[\left\|M_k-M_{k-1}\right\|_X^p\right]}{(2^{p-1}-1)K_p(X)^p}.
\]
\end{theorem}

We shall also need the following variant of Pisier's inequality:

\begin{corollary}\label{coro L_2(X)}
For every $p,q\in (1,\infty)$ there is a constant $c(p,q)\in
(0,\infty)$ with the following properties. Let $X$ be a normed space
with $K_p(X)<\infty$. Then for every martingale
$\{M_k\}_{k=0}^n\subseteq X$, if $p\le q$, then
\begin{equation}\label{eq:easy figiel}
\E \left[\left\|M_n-M_0\right\|_X^q\right]\ge
\frac{c(p,q)}{K_p(X)^p}\sum_{k=1}^n\E\left[\left\|M_k-M_{k-1}\right\|_X^q\right].
\end{equation}
and if  $p\ge q$, then
\begin{equation}\label{eq:hard figiel}
\E \left[\left\|M_n-M_0\right\|_X^q\right]\ge
\frac{c(p,q)}{n^{1-\frac{q}{p}}K_p(X)^q}\sum_{k=1}^n\E\left[\left\|M_k-M_{k-1}\right\|_X^q\right].
\end{equation}
\end{corollary}

\begin{proof}
\ifsodaelse{
  Assume first that $p\le q$. Observe that by the definition of $K_p(X)$
  for every $\e\in [0,2]$ we have
  \begin{multline*}
  \delta_X(\e)\ge 1-\left(1-\frac{\e^p}{2^pK_p(X)^p}\right)^{1/p}
  \ifsodaelse{\\}{}
  \ge
  \frac{\e^p}{p2^pK_p(X)^p}\ge \frac{2^{p-q}\e^q}{p2^pK_p(X)^p}.
  \end{multline*}
  By (the proof of) Proposition~7
  In~\cite{BCL}, it follows that $K_q(X)\le C K_p(X)^{p/q}$, where $C$
  depends only on $p,q$. Hence~\eqref{eq:easy figiel} follows from
  Theorem~\ref{thm:pisier ineq}.

  Assume next that $q\le p$.
}
{Here we only prove the case $q \le p$, since it is the only one needed in the sequel.}
  By a theorem of Figiel~\cite{Fiegel76} (combined
with~\cite[Proposition 7]{BCL}) we have $K_p(L_q(X))\le CK_p(X)$,
where $C$ depends only on $p,q$. We may therefore
apply~\eqref{eq:two point convex} to the vectors
\begin{align*}
x& =M_{n-1}-M_0+\frac{M_n-M_{n-1}}{2}\in L_q(X),\\ y &=\frac{M_n-M_{n-1}}{2}\in
L_q(X)
\end{align*}
 to get
\begin{multline}\label{eq:prove lin}
2\left(\E\left[\left\|M_{n-1}-M_0+\frac{M_n-M_{n-1}}{2}\right\|_X^q\right]\right)^{p/q}
 +\frac{2}{2^pCK_p(X)^p}\left(\E\left[\left\|M_n-M_{n-1}\right\|_X^q\right]\right)^{p/q}\\\le
\left(\E\left[\left\|M_n-M_{0}\right\|_X^q\right]\right)^{p/q}+\left(\E\left[\left\|M_{n-1}-M_{0}\right\|_X^q\right]\right)^{p/q}.
\end{multline}
Assume that $\{M_k\}_{k=0}^n\subseteq X$ is a martingale with
respect to the filtration $\mathcal{F}_0\subseteq
\mathcal{F}_1\subseteq \cdots\subseteq \mathcal{F}_{n-1}$. Then
\begin{equation*}
\E\left[\left\|M_{n-1}-M_{0}\right\|_X^q\right]
=\E\left[\left\|M_{n-1}-M_{0}+\E\left[\left.M_n-M_{n-1}\right|\mathcal{F}_{n-1}\right]\right\|_X^q\right]
\ifsodaelse{\\}{}
\le
\E\left[\left\|M_{n}-M_{0}\right\|_X^q\right],
\end{equation*}
and
\begin{multline*}
\E\left[\left\|M_{n-1}-M_{0}\right\|_X^q\right]
\ifsodaelse{\\}{}
=\E\left[\left\|M_{n-1}-M_{0}+\E\left[\left.\frac{M_n-M_{n-1}}{2}\right|\mathcal{F}_{n-1}\right]\right\|_X^q\right]\\\le
\E\left[\left\|M_{n-1}-M_0+\frac{M_n-M_{n-1}}{2}\right\|_X^q\right].
\end{multline*}
Thus~\eqref{eq:prove lin} implies that
\begin{equation}\label{eq:before sum q}
\left(\E\left[\left\|M_{n-1}-M_0\right\|_X^q\right]\right)^{p/q}+\frac{\left(\E\left[\left\|M_n-M_{n-1}\right\|_X^q\right]\right)^{p/q}}{2^pCK_p(X)^p}
 \le
\left(\E\left[\left\|M_n-M_{0}\right\|_X^q\right]\right)^{p/q}.
\end{equation}
Applying~\eqref{eq:before sum q} inductively we get the lower bound
\begin{multline*}
2^pCK_p(X)^p\left(\E\left[\left\|M_n-M_{0}\right\|_X^q\right]\right)^{p/q}
 \ge
\sum_{k=1}^n\left(\E\left[\left\|M_k-M_{k-1}\right\|_X^q\right]\right)^{p/q}
\ifsodaelse{\\}{}
\\\ge
\frac{1}{n^{\frac{p}{q}-1}}\left(\sum_{k=1}^n\E\left[\left\|M_k-M_{k-1}\right\|_X^q\right]\right)^{p/q},
\end{multline*}
which is precisely~\eqref{eq:hard figiel}.
\end{proof}

Note that Corollary~\ref{coro L_2(X)} specialized to $q=2$ is
precisely~\eqref{eq:pisier variant}.

\subsection{A metric space without uniform decay of the Poincar\'e constant}
\label{sec:no-decay} Observe that as long as a metric space
$(X,d_X)$ contains at least two points, the fact that
$\bpconst(A,d_X^p)<\infty$ implies that $A$ is ergodic, and
therefore $\lim_{t\to \infty }\bpconst(A^t,d_X^p)=\lim_{t\to \infty
}\bpconst(\A_t(A),d_X^p)=1$. However this weak decay is insufficient
for our purposes. For the iterative construction in this paper we
need at the very least the following type of {\em uniform decay} of
the Poincar\'e constant:
\begin{multline*}
 \forall M>1\, \exists t\in \mathbb N\, \exists s_0 \in [1,\infty)\,
 \forall n\in \N \ifsodaelse{\\}{\,}\forall A\in M_n(\R)
 \  \text{symmetric\ \&\  stochastic},\\
  \bpconst(A,d_X^p)\ge s_0 \implies
\bpconst(\A_t(A),d_X^p)\le \frac{\bpconst(A,d_X^p)}{M} .
\end{multline*}

Here we show that there exists a  metric space with no such uniform decay.
\begin{proposition}
There exists a metric space $(X,\rho)$, and  graphs
$\{G_n=(V_n,E_n)\}_{n=1}^\infty$,  where $|V_n|=n$, such that
$\lim_{n\to \infty}\bpconst(G_n,\rho^2)=\infty$, and for every $t\in
\mathbb N$, there exists $N_0$ such that for every $n>N_0$,
$\bpconst(\A_t(G_n),\rho^2)\ge \frac12\bpconst(G_n,\rho)$.
\end{proposition}

\begin{proof}
Let $G_n=(V_n,E_n)$ be an arbitrary family of constant degree
expanders, i.e., $G_n$ is a degree $d$ (say $d=4$), $n$-vertex graph
satisfying $\bpconst(G_n,\|\cdot\|_2^2)\le C$, for some $C>0$. Let
$(X,\rho)$ be the metric space defined as follows: as a set,
 $X=\ell_\infty \cap \mathbb Z^{\aleph_0}$, i.e., the set of all finite integer-valued sequences. The metric $\rho$
 on $X$ is given by  $\rho(x,y)= \log(1+\|x-y\|_\infty)$. Note that $\rho$ is indeed
 a metric since $T(s)= \log(1+s)$ is concave, increasing and $T(0)=0$ (i.e., $\rho=T\circ
 \|\cdot\|_\infty$ is a metric transform of the $\ell_\infty$ norm).

 First we prove that $\bpconst(G_n,\rho^2)\lesssim (\log(1+\log n))^2$.
 Let $f,g:G_n\to X$ be  arbitrary mappings. Our goal is to prove that
 \begin{equation*}
  \frac1{n^2} \sum_{u,v\in V} \rho(f(u),g(v))^2
  \ifsodaelse{\\}{}
  \lesssim \frac{\left(\log (1+\log n)\right)^2}{nd}\sum_{\{u,v\}\in E} \rho(f(u),f(v))^2 .
  \end{equation*}
Write $S=f(V)\cup g(V)\subseteq\mathbb Z^\infty$. Let
$B:(S,\|\cdot\|_\infty)\to \ell_2$ be the Bourgain
embedding~\cite{Bourgain-embed}
 of $S$ equipped the $\ell_\infty$ norm into $\ell_2$, i.e., for every $u,v\in
 V$,
  \begin{equation} \label{eq:limitations-1}
  \|f(u)-g(v)\|_\infty \le \|B(f(u))-B(g(v))\|_2
  \ifsodaelse{\\}{}
  \le c(1+\log n)  \|f(u)-g(v)\|_\infty .
  \end{equation}
  for some constant $c\ge 1$.

  We next apply the metric transform $T(x)=\log (1+x)$ to~\eqref{eq:limitations-1}, and obtain
  for $f(u)\ne g(v)$,
  \begin{align}
\nonumber  \rho \ifsodaelse{&}{} (f(u),g(v))
\ifsodaelse{\\ \nonumber}{}
& = T(\|f(u)-g(v)\|_\infty)\\
&\le T\left(\|B(f(u))-B(g(v))\|_2\right)\nonumber \\  \nonumber
\nonumber   & \le \log \left(1+ c(1+\log n) \|f(u)-g(v)\|_\infty \right) \\
\nonumber   &  \le \log(c(1+\log n)) + \log(2\|f(u)-g(v)\|_\infty) \\
  &\lesssim \log(1+\log n) \cdot \rho(f(u),g(v)).
  \label{eq:limitations-2}
  \end{align}
  The second and last inequalities above follow from the fact that $\|f(u)-g(v)\|_\infty \ge 1$.

  As was shown in~\cite[Remark~5.4]{MN-quotients}, there exists a universal
  constant $c'>1$ and a mapping $\phi:\ell_2\to \ell_2$ such that
  for  all $x,y\in \ell_2$ we have $T\left(\|x-y\|_2\right)\le
  \|\phi(x)-\phi(y)\|_2\le c'T\left(\|x-y\|_2\right)$. This fact, together with~\eqref{eq:limitations-2}, implies that
  the mapping $\psi=\phi\circ B:S\to \ell_2$ satisfies:
  \begin{equation*}
   \rho(f(u),g(v)) \le \|\psi(f(u))-\psi(g(v))\|_2
   \ifsodaelse{\\}{}
   \lesssim \log(1+\log n) \cdot \rho(f(u),g(v)) ,
   \end{equation*}
  for every $u,v\in V$.
  Applying the Poincar\'e inequality of $G$ in $\ell_2$ we conclude that
  \begin{multline*}
  \frac1{n^2} \ifsodaelse{&}{} \sum_{u,v\in V} \rho(f(u),g(v))^2
  \ifsodaelse{\\}{}
    \le
  \frac1{n^2} \sum_{u,v\in V} \|\psi(f(u))-\psi(g(v))\|_2^2 \displaybreak[1]\\
   \lesssim
   \frac{1}{nd}\sum_{(u,v)\in E} |\psi(f(u))-\psi(g(v))\|_2^2
   \lesssim \frac{(\log(1+\log n))^2}{nd}\sum_{(u,v)\in E}  \rho(f(u),g(v))^2.
  \end{multline*}
  This proves that $\bpconst(G_n,\rho^2)\lesssim (\log (1+\log n))^2$.

Next, we show a lower bound on $\bpconst(\A_t(G),\rho)$. For this
purpose it is sufficient to examine a specific embedding of
$\A_t(G)$ in $X$. Let $g=f:V\to \mathbb Z^{\aleph_0}$ which is an
isometric embedding of the shortest path metric on $\A_t(G)$ into
$(\mathbb Z^{\aleph_0},\|\cdot\|_\infty)$. In this case for every
$\{u,v\}\in E(\A_t(G))$,
$\rho(f(u),f(v))=T(\|f(u)-f(v)\|_\infty)=T(1)=1$. On the other hand,
since the degree of $\A_t(G)$ is $td^t$, at least half of the pairs
in $V\times V$ are at distance $\gtrsim \frac{\log n}{t \log d}$ in
the shortest path metric metric on $\A_t(G)$, and hence their images
under $f$ are at $\rho$-distance $\ge\log\left(1+c''\frac{\log n}{t
\log d}\right)$, where $c''>0$ is a universal constant. Hence,
provided $t\le c'''\log n$ for some sufficiently small constant
$c'''$, we have $\bpconst(\A_t(G),\rho^2) \gtrsim\left(\log
\left(1+\frac{\log n}{t}\right)\right)^2$.
\end{proof}

\subsection{Construction of the base graph}
\label{sec:base-details}

\ifsodaelse{}{
\begin{proof}[Proof of Claim~\ref{cl:real-beckner}]
The claim easily follows by computing the right hand side  in~\eqref{eq:real-beckner-b} when  $f=W_A$. We write
$x=x_Ax_{A^c}$, and
\begin{align*}
 \sum_{y\in\,\mathbb F_2^n}
 \ifsodaelse{\!\!\!\!}{}
 &
 \ifsodaelse{\,\,\,\,}{}
 \left( \ifsodaelse{\tfrac}{\frac}{1-e^{-t}}2 \right)^{\|x-y\|_1} \left( \ifsodaelse{\tfrac}{\frac}{1+e^{-t}}2 \right)^{n-\|x-y\|_1}
{\prod}_{i\in A}(-1)^{y_i}
\\ &   =
\ifsodaelse{\!\!\!\!}{}  \sum_{y_A\in\, \mathbb F_2^A}
  \ifsodaelse{\!\!\!}{}
  \left( \ifsodaelse{\tfrac}{\frac}{1-e^{-t}}2 \right)^{\|x_A-y_A\|_1}
  \ifsodaelse{\!\!}{}
  \left( \ifsodaelse{\tfrac}{\frac}{1+e^{-t}}2 \right)^{|A|-\|x_A-y_A\|_1}
\prod_{i\in A}(-1)^{y_i}
  \\ &  \qquad \ifsodaelse{\ }{\qquad} \cdot \ifsodaelse{\!\!\!\!\!\!\!\!\!}{}
  \sum_{y_{A^c}\in\, \mathbb F_2^{[n]\setminus A}}
  \ifsodaelse{\!\!\!}{}
  \left( \ifsodaelse{\tfrac}{\frac}{1-e^{-t}}2 \right)^{\|x_{A^c}-y_{A^c}\|_1}
  \ifsodaelse{\!\!}{}
  \left( \ifsodaelse{\tfrac}{\frac}{1+e^{-t}}2 \right)^{n-|A|-\|x_{A^c}-y_{A^c}\|_1}
  \\ & = \sum_{\ell=0}^{|A|} \ifsodaelse{\tbinom}{\binom}{|A|}{\ell} \left( \ifsodaelse{\tfrac}{\frac}{1-e^{-t}}2 \right)^\ell
  \left( \ifsodaelse{\tfrac}{\frac}{1+e^{-t}}2 \right)^{|A|-\ell} (-1)^\ell W_A(x)
  \\ &= e^{-t|A|} W_A(x).  \qedhere
\end{align*}
\end{proof}
}

\begin{proof}[Proof of Proposition~\ref{prop:linear=>nonlin}]
Let $f,g:[N]\to X$, denote by $\bar f=\Avg f$, and $\bar g=\Avg g$,
$\widetilde f= f- \bar f$, $\widetilde g= g- \bar g$. Then
$\sum_{i=1}^N \widetilde f(i)=\sum_{i=1}^N \widetilde g(i)=0$, and
so by the definition of $\lambda$, $\|A \widetilde f\|_{L_p(X)} \le
\lambda \| \widetilde f\|_{L_p(X)}$, and $\|A \widetilde
g\|_{L_p(X)} \le \lambda \| \widetilde g\|_{L_p(X)}$. Hence

\begin{multline*}
\|\left( \begin{smallmatrix} 0 & A\\ A & 0 \end{smallmatrix}
\right)(\widetilde f \oplus \tilde g)\|_{L_p(X)}^p =
\frac{\|A\widetilde g\|_{L_p(X)}^p + \|A\widetilde f\|_{L_p(X)}^p}{2} \\
\le \lambda^p \frac{\|\widetilde f\|_{L_p(X)}^p + \|\widetilde
g\|_{L_p(X)}^p}{2} = \lambda^p \|\widetilde f\oplus \widetilde
g\|_{L_p(X)}^p .
\end{multline*}
It follows that,
\begin{equation*}
 (1-\lambda) \|\widetilde f \oplus \widetilde g\|_{L_p(X)}
 \ifsodaelse{\\}{}
 \le \|\widetilde f \oplus \widetilde g\|_{L_p(X)} - \|\left(\begin{smallmatrix} 0 & A\\A & 0 \end{smallmatrix}\right)(\widetilde f\oplus \widetilde g)\|_{L_p(X)}
 \le \left\|\left(\begin{smallmatrix}I &-A\\ -A & I \end{smallmatrix} \right) (\widetilde f \oplus \widetilde g)\right\|_{L_p(X)}. \end{equation*}
So,
\begin{align*}
\frac{1}{N^2}
\ifsodaelse{&}{}
 \sum_{i,j\in [N]}
\ifsodaelse{}{&}
  \|f(i)- g(j)\|_X^p \le 2^{p-1}\|\bar f - \bar g\|_{X}^p +
 \frac{2^{p-1}}{N^2} \sum_{i,j\in [N]} \|\widetilde f(i)- \widetilde g(j)\|_X^p \displaybreak[1]\\
& \le 2^{p-1}\|\bar f - \bar g\|_{X}^p + 4^{p-1} ( \|\widetilde f\|_{L_p(X)}^p+
\|\widetilde g\|_{L_p(X)}^p)\\
&= 2^{p-1}\|\bar f - \bar g\|_{X}^p + 2^{2p-1} \|\widetilde f \oplus \widetilde g\|_{L_p(X)}^p \displaybreak[1]\\
&\le 2^{p-1}\|\bar f - \bar g\|_{X}^p
\ifsodaelse{\\&\quad}{}
+ 2^{2p-1} (1-\lambda)^{-p} \left\|\left(\begin{smallmatrix}I &-A\\ -A & I \end{smallmatrix} \right) (\widetilde f\oplus \widetilde g)\right\|_{L_p(X)}^p.
\end{align*}
Expanding the last term,
\begin{align*}
\Bigl\|& \left(\begin{smallmatrix}I &-A\\ -A & I \end{smallmatrix} \right)  (\widetilde f\oplus \widetilde g)\Bigr\|_{L_p(X)}^p \\
&= \frac{1}{2N} \sum_{i\in [N]} \biggl \| \sum_{j\in [N]} A_{ij} \left(\widetilde f(i)-\widetilde g(j) \right) \biggr\|_X^p
\ifsodaelse{\\ & \quad}{}
+ \frac{1}{2N}  \sum_{i\in [N]} \biggl \| \sum_{j\in [N]} A_{ij} \left(\widetilde g(i)-\widetilde f(j) \right) \biggr\|_X^p
\\
\ifsodaelse{}{
  &\le \frac1N  \sum_{i,j\in [N]} A_{ij} \left \|  \widetilde f(i)-\widetilde g(j) \right\|_X^p \\
}
& \le 2^{p-1}\|\bar f - \bar g\|_X^p+ \frac{2^{p-1}}{N} \sum_{i,j\in [N]} A_{ij}  \|   f(i)- g(j) \|_X^p.
\end{align*}

We next observe that for any permutation $\pi:[N]\to [N]$ (and $p\ge 1$),
\( \|\bar f - \bar g\|^p_X \le \frac1N\sum_i \|f(i)-g(\pi(i))\|^p_X .\)
Since $A$ is symmetric stochastic, it is in the convex hull of the permutation matrices, so we can write
$A=\sum_\ell \alpha ^\ell \Pi_\ell$,  where $\{\Pi_\ell\}_\ell$ are permutation matrices, and $\alpha_\ell \ge 0$, $\sum_\ell \alpha_\ell=1$.
Hence,
\begin{equation*}
 \|\bar f -\bar g\|^p_X\le \sum_{\ell} \alpha_\ell \frac1N \sum_{i,j} \Pi^\ell_{ij} \|f(i)-g(j)\|^p_X
\ifsodaelse{\\}{}
 =
\frac1N\sum_{i,j} A_{ij} \|f(i)-g(j)\|^p_X .
\end{equation*}

We conclude that,
\begin{equation*}
 \frac1{N^2} \sum_{i,j\in [N]} \|f(i)- g(j)\|_X^p
 \ifsodaelse{\\}{}
 \le
\frac{8^p (1-\lambda)^{-p}}{N}  \sum_{i,j\in[N]} A_{ij} \|f(i)-g(j)\|_X^p. \qedhere
\end{equation*}
\end{proof}

\subsubsection{Proof of Lemma~\ref{lem:bounding-beckner}}
\label{sec:proof-of-bounding-beckner}

The {\em Rademacher projection} of $f$ is defined by
\[
R_1 f=\sum_{j=1}^n \widehat f(A)W_{\{j\}}.
\]
The $K$-convexity constant of $X$, denoted $K(X)$, is the smallest
constant $K$ such that for every $n$ and every $f:\F_2^n\to
X$,
\ifsodaelse{\(}{\[}
 \|(R_1 f)\|_{L_2(X)}^2  \le K^2  \|f\|_{L_2(X)}^2.
\ifsodaelse{\)}{\]}
We also define the $k$th level Rademacher projection by
\[ R_k f= \sum_{\substack{A\subseteq [n]\\ |A|=k}} \hat f(A) W_A .\]


We assume throughout that $X$ is a complex Banach space. The case of
Banach spaces over $\R$ follows formally from this case by
complexification. Extend the heat semi-group $T_z:L_p(X) \to L_p(X)$
to $z\in \mathbb C$,  by setting:
\[ T_z f = \sum_{A\subseteq [n]} e^{-z|A|} \hat f(A) W_A .\]
Another way to write it is $T_z=\sum_{k=0}^n e^{-zk} R_k$. We shall
use the following deep theorem of Pisier\cite{Pisier-K-convex}:

\begin{theorem}[{\cite{Pisier-K-convex}}]\label{thm:Pis K-conv}
For $K$-convex $X$, and $p\in(1,\infty)$, there exists $\phi,M>$
such that if $|\arg z|<\phi$ then $\|T_z\|_{L_p(X)\to L_p(x)}\le M$.
\end{theorem}
See~\cite{Mau03} for bounds on $\phi, M$ in terms of $K(X)$. We will
require the following standard corollary of Theorem~\ref{thm:Pis
K-conv}:

\begin{corollary}
 $\|R_k\|_{L_p(X)\to L_p(X)} \le M e^{ak}$.
\end{corollary}
\begin{proof}
Define $a=\frac{\pi}{\tan \phi}$, so that all the points in the
segment $\{(a,iy):\; y\in (-\pi,\pi)\}$ have argument at most
$\phi$. Calculate,
\begin{multline*}
\frac{1}{2\pi} \int_{-\pi}^\pi e^{ikt} T_{a+it} dt = \frac{1}{2\pi}
\int_{-\pi}^\pi e^{ikt} \sum_{r=0}^n e^{-(a+it)r} R_r dt
\\
= \frac{1}{2\pi} \sum_{r=0}^n e^{-ra}R_r \int _{-\pi}^\pi e^{i(k-r)t}dt
= e^{-ka} R_k,
\end{multline*}
and therefore, by convexity,
\begin{equation*}
 \|R_k\|_{L_p(X)\to L_p(X)}
 \ifsodaelse{\\}{}
 \le \frac{e^{ka}}{2\pi} \int_{-\pi}^\pi |e^{ikt}|
 \ifsodaelse{\!}{}
 \cdot
 \ifsodaelse{\!}{}
 \|T_{a+it}\|_{L_p(X)\to L_p(X)}
dt
\le Me^{ka} .  \qedhere \end{equation*}
\end{proof}

\begin{corollary}\label{coro:far away}
Assume that $\Re z\ge 2a$. Then
 \begin{multline*}
 \|T_z\|_{L_p^{\ge m}(X)\to L_p^{\ge m}(X)}
 \ifsodaelse{\\}{}
 =
 \left\|\sum_{k\ge m} e^{-zk} R_k \right\|_{L_p^{\ge m}(X)\to L_p^{\ge m}(X)} \\
 \le
 \sum_{k\ge m} e^{-2ak}\cdot Me^{ak} = \frac{M}{1-e^{-a}} e^{-ma}. \qquad \qed
 \end{multline*}
 \end{corollary}

The ensuing argument is a quantitative version of the proof of the
main theorem of Pisier in~\cite{pisier-2007}. Let $r=2\sqrt{a^2+\pi^2}$, and define $
V=\left\{z\in \mathbb C:\ |z|\le r\ \wedge\ |\arg z|\le \phi\right\}$ (the set $V\subseteq \mathbb C$ is depicted in Figure~\ref{fig:triangle}). Denote
 $V_0=\{ x\pm i x\tan \phi:\; x\in[0,2a) \}$ and  $V_1=\{re^{i\theta}:\; |\theta|\le \phi\}$, so that we have the disjoint union $\partial V=V_0\cup V_1$.

 \begin{figure}[ht]
 \begin{center}
 \ifsodaelse{
   \includegraphics[scale=0.8]{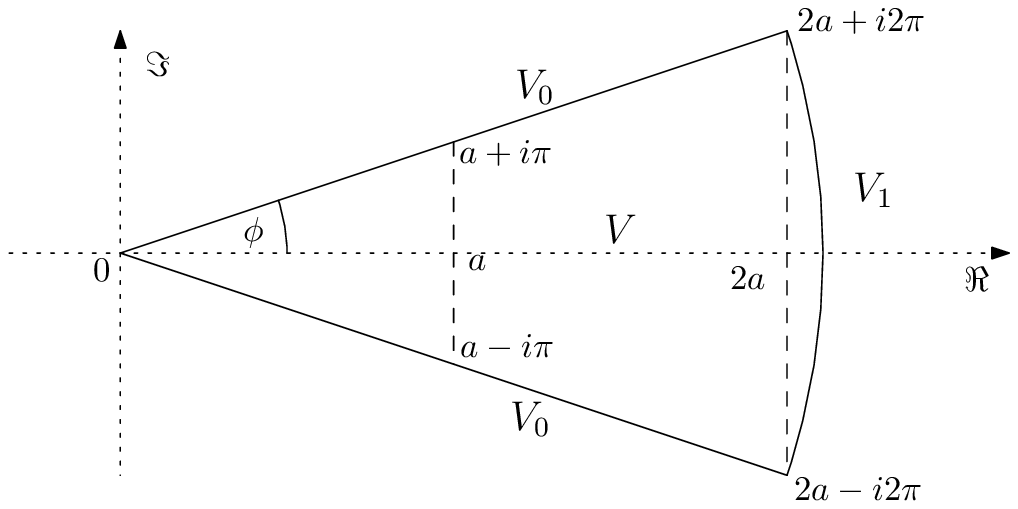}
 }
 {
   \includegraphics{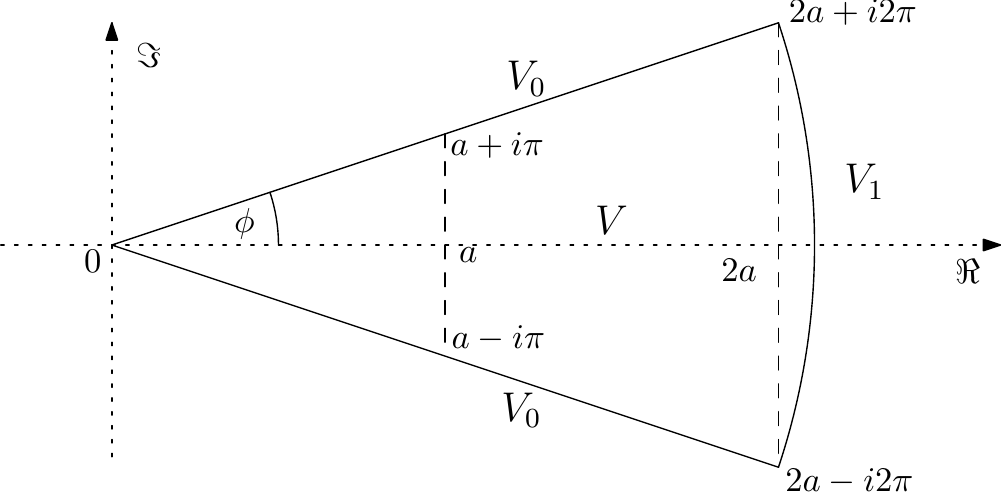}
 } \end{center}
 \caption{The sector $V$.}
 \label{fig:triangle}
 \end{figure}

Fix $t\in(0,2a)$. We shall use the Harmonic measure corresponding to
$V$, i.e., the Borel probability measure $\mu=\mu_t$ on $\partial V$
such that for every bounded analytic function on $V$,
\begin{equation}\label{eq:harmonic}
f(t) =\int_{\partial V} f(z) d\mu(z) .\end{equation}
We refer to~\cite{GM08} for more information on this topic and the
ensuing discussion. It suffices to say here that for a Borel set
$U\subseteq
\partial V$, $\mu_t(U)$ is the probability that Brownian motion
starting at $t$ leaves $V$ at $U$. Equivalently, by conformal
invariance, $\mu_t$ is the push-forward of the normalized Lebesgue
measure on the unit circle $S^1$ under the Riemann mapping from the
unit disk to $V$ which takes the origin to $t$.

Let $\theta=\mu(V_1)$, and denote $\mu=(1-\theta)\mu_0+\theta
\mu_1$, where $\mu_i$ is a probability measure on $V_i$,
$i\in\{0,1\}$. We will require the following standard bound on $\theta$:
\begin{lemma}\label{lem:theta bound} We have
$\theta\asymp
\left(\frac{t}{r}\right)^{\frac{\pi}{2\phi}}$.
\end{lemma}
\begin{proof}
The proof is a simple exercise in conformal invariance. Let $\mathbb D$ denote the unit disk centered at the origin, and let $\mathbb D_+$ denote the intersection of $\mathbb D$ with the half plane $\arg z\ge 0$. The mapping $z\mapsto  \left(\frac{z}{r}\right)^{\frac{\pi}{2\phi}}$ is a conformal equivalence between $V$ and $\mathbb D_+$. The mapping $z\mapsto -i\frac{z+i}{z-i}$ is a conformal equivalence between $\mathbb D_+$ and the positive quadrant $\{x+iy:\ x,y\ge 0\}$, and the mapping $z\mapsto z^2$ is a conformal equivalence between the positive quadrant and the upper half-plane. Finally, the mapping $z\mapsto \frac{z-i}{z+i}$ is a conformal equivalence between the upper half-plane and $\mathbb D$. By composing these mappings, we obtain the following conformal equivalence between $V$ and $\mathbb D$:
$$
F(z)=\frac{-\left(\left(\frac{z}{r}\right)^{\frac{\pi}{2\phi}}+i\right)^2
-i\left(\left(\frac{z}{r}\right)^{\frac{\pi}{2\phi}}-i\right)^2}
{-\left(\left(\frac{z}{r}\right)^{\frac{\pi}{2\phi}}+i\right)^2+i
\left(\left(\frac{z}{r}\right)^{\frac{\pi}{2\phi}}-i\right)^2}.
$$
Thus, by composing with the conformal automorphism of $\mathbb D$, $z\mapsto \frac{z-F(t)}{1-\overline{F(t)}z}$, we obtain the mapping $G(z)= \frac{F(z)-F(t)}{1-\overline{F(t)}F(z)}$, which is a conformal equivalence between $V$ and $\mathbb D$ with $G(t)=0$.

By conformal invariance, $\theta$ is the length of the arc $G(V_1)\subseteq \partial \mathbb D=S^1$, divided by $2\pi$. The assertion of Lemma~\ref{lem:theta bound} now follows from a direct computation.
\end{proof}

\begin{lemma}\label{lem:strip}
For every $\e\in (0,1)$ there exists a bounded analytic function  $\Psi:V\to \mathbb C$, such that
$\Psi(t)=1$, for every $z\in V_0$, $|\Psi(z)|\le \e$, and for every $z\in V_1$, $|\Psi(z)|\le \e^{-\theta^{-1}+1}$.
\end{lemma}

\begin{proof} The proof is the same as the proof of Claim 2 in~\cite{pisier-2007}. Consider the strip $S=\{z\in \mathbb C:\ \Re(z)\in [0,1]\}$ and for $j\in \{0,1\}$ let $S_j=\{z\in \mathbb C:\ \Re(z)=j\}$. The argument in Claim 1 in~\cite{pisier-2007} shows that there exists a a conformal equivalence $h:V:\to S$ such that $h(t)=\theta$, $h(V_0)=S_0$, $h(V_1)=S_1$. Now define $\Psi(z)=\e^{1-\theta^{-1}h(z)}$.
\end{proof}

We next calculate
\begin{equation*}
T_t =\Psi(t)T_t= \int_{\partial V} \Psi(z) T_z d \mu(z)
= (1-\theta) \int_{V_0}\Psi(z) T_zd\mu_0(z) +
\theta \int_{V_1} \Psi(z) T_z d\mu_1(z).
\end{equation*}
Hence, using Lemma~\ref{lem:strip} with $\e=e^{-am\theta}$, we get the bound:
\begin{equation}\label{eq:small t}
\|T_t\|_{L_p^{\ge m}(X)\to L_p^{\ge m}(X)}
\ifsodaelse{\\}{}
\le (1-\theta) \e M+ \theta \e^{1-\theta^{-1}} C e^{-am}\le (M+C)e^{-am\theta} .
\end{equation}

Since by Lemma~\ref{lem:theta bound} we have $\theta\asymp_{\phi} t^{\frac{\pi}{2\phi}}$ (where the implied constant depends on $\phi$), the bound~\eqref{eq:small t}, which holds for $t\in (0,2a)$, combined with Corollary~\ref{coro:far away}, concludes the proof of Lemma~\ref{lem:bounding-beckner}.\hfill \qed

\bibliographystyle{abbrvurl}
\bibliography{zigzag}

\begin{thebibliography}{10}

\bibitem{Ball}
K.~Ball.
\newblock Markov chains, {R}iesz transforms and {L}ipschitz maps.
\newblock {\em Geom. Funct. Anal.}, 2(2):137--172, 1992.
\newblock \href {http://dx.doi.org/10.1007/BF01896971}
  {\path{doi:10.1007/BF01896971}}.

\bibitem{BCL}
K.~Ball, E.~A. Carlen, and E.~H. Lieb.
\newblock Sharp uniform convexity and smoothness inequalities for trace norms.
\newblock {\em Invent. Math.}, 115(3):463--482, 1994.
\newblock \href {http://dx.doi.org/10.1007/BF01231769}
  {\path{doi:10.1007/BF01231769}}.

\bibitem{BLMN05}
Y.~Bartal, N.~Linial, M.~Mendel, and A.~Naor.
\newblock On metric ramsey-type phenomena.
\newblock {\em Annals of Mathematics}, 162(2):643--709, 2005.
\newblock \href {http://arxiv.org/abs/math.MG/0406353}
  {\path{arXiv:math.MG/0406353}}.

\bibitem{Bourgain-embed}
J.~Bourgain.
\newblock On {L}ipschitz embedding of finite metric spaces in {H}ilbert space.
\newblock {\em Israel J. Math.}, 52(1-2):46--52, 1985.
\newblock \href {http://dx.doi.org/10.1007/BF02776078}
  {\path{doi:10.1007/BF02776078}}.

\bibitem{BKL07}
B.~Brinkman, A.~Karagiozova, and J.~R. Lee.
\newblock Vertex cuts, random walks, and dimension reduction in series-parallel
  graphs.
\newblock In {\em S{TOC}'07---{P}roceedings of the 39th {A}nnual {ACM}
  {S}ymposium on {T}heory of {C}omputing}, pages 621--630. ACM, New York, 2007.
\newblock \href {http://dx.doi.org/10.1145/1250790.1250882}
  {\path{doi:10.1145/1250790.1250882}}.

\bibitem{Fiegel76}
T.~Figiel.
\newblock On the moduli of convexity and smoothness.
\newblock {\em Studia Math.}, 56:121--155, 1976.

\bibitem{GM08}
J.~B. Garnett and D.~E. Marshall.
\newblock {\em Harmonic measure}, volume~2 of {\em New Mathematical
  Monographs}.
\newblock Cambridge University Press, Cambridge, 2008.
\newblock Reprint of the 2005 original.

\bibitem{Gromov-filling}
M.~Gromov.
\newblock Filling {R}iemannian manifolds.
\newblock {\em J. Differential Geom.}, 18(1):1--147, 1983.

\bibitem{Gro93}
M.~Gromov.
\newblock Asymptotic invariants of infinite groups.
\newblock In {\em Geometric group theory, {V}ol.\ 2 ({S}ussex, 1991)}, volume
  182 of {\em London Math. Soc. Lecture Note Ser.}, pages 1--295. Cambridge
  Univ. Press, Cambridge, 1993.

\bibitem{Gromov-random-group}
M.~Gromov.
\newblock Random walk in random groups.
\newblock {\em Geom. Funct. Anal.}, 13(1):73--146, 2003.
\newblock \href {http://dx.doi.org/10.1007/s000390300002}
  {\path{doi:10.1007/s000390300002}}.

\bibitem{HLW}
S.~Hoory, N.~Linial, and A.~Wigderson.
\newblock Expander graphs and their applications.
\newblock {\em Bulletin of the AMS}, 43(4):439--561, 2006.
\newblock \href {http://dx.doi.org/10.1090/S0273-0979-06-01126-8}
  {\path{doi:10.1090/S0273-0979-06-01126-8}}.

\bibitem{Jam74}
R.~C. James.
\newblock A nonreflexive {B}anach space that is uniformly nonoctahedral.
\newblock {\em Israel J. Math.}, 18:145--155, 1974.

\bibitem{Jam78}
R.~C. James.
\newblock Nonreflexive spaces of type {$2$}.
\newblock {\em Israel J. Math.}, 30(1-2):1--13, 1978.

\bibitem{JL75}
R.~C. James and J.~Lindenstrauss.
\newblock The octahedral problem for {B}anach spaces.
\newblock In {\em Proceedings of the {S}eminar on {R}andom {S}eries, {C}onvex
  {S}ets and {G}eometry of {B}anach {S}paces ({M}at. {I}nst., {A}arhus {U}niv.,
  {A}arhus, 1974; dedicated to the memory of {E}. {A}splund)}, pages 100--120.
  Various Publ. Ser., No. 24, Aarhus Univ., Aarhus, 1975. Mat. Inst.

\bibitem{KY06}
G.~Kasparov and G.~Yu.
\newblock The coarse geometric {N}ovikov conjecture and uniform convexity.
\newblock {\em Adv. Math.}, 206(1):1--56, 2006.
\newblock \href {http://dx.doi.org/10.1016/j.aim.2005.08.004}
  {\path{doi:10.1016/j.aim.2005.08.004}}.

\bibitem{KN06}
S.~Khot and A.~Naor.
\newblock {Nonembeddability theorems via Fourier analysis}.
\newblock {\em Mathematische Annalen}, 334(4):821--852, 2006.
\newblock \href {http://arxiv.org/abs/math/0510547}
  {\path{arXiv:math/0510547}}, \href
  {http://dx.doi.org/10.1007/s00208-005-0745-0}
  {\path{doi:10.1007/s00208-005-0745-0}}.

\bibitem{Lafforgue}
V.~Lafforgue.
\newblock Un renforcement de la propri\'et\'e ({T}).
\newblock {\em Duke Math. J.}, 143(3):559--602, 2008.
\newblock \href {http://dx.doi.org/10.1215/00127094-2008-029}
  {\path{doi:10.1215/00127094-2008-029}}.

\bibitem{lafforgue-2009}
V.~Lafforgue.
\newblock Propri\'et\'e ({T}) renforc\'ee {B}anachique et transformation de
  {F}ourier rapide.
\newblock Preprint, 2009.

\bibitem{LLR}
N.~Linial, E.~London, and Y.~Rabinovich.
\newblock The geometry of graphs and some of its algorithmic applications.
\newblock {\em Combinatorica}, 15(2):215--245, 1995.
\newblock \href {http://dx.doi.org/10.1007/BF01200757}
  {\path{doi:10.1007/BF01200757}}.

\bibitem{Mat97}
J.~Matou\v{s}ek.
\newblock On embedding expanders into $\ell_p$ spaces.
\newblock {\em Israel J. Math.}, 102:189--–197, 1997.
\newblock \href {http://dx.doi.org/10.1007/BF02773799}
  {\path{doi:10.1007/BF02773799}}.

\bibitem{Mau03}
B.~Maurey.
\newblock Type, cotype and {$K$}-convexity.
\newblock In {\em Handbook of the geometry of {B}anach spaces, {V}ol.\ 2},
  pages 1299--1332. North-Holland, Amsterdam, 2003.

\bibitem{MN-quotients}
M.~Mendel and A.~Naor.
\newblock Euclidean quotients of finite metric spaces.
\newblock {\em Adv. Math.}, 189(2):451--494, 2004.
\newblock \href {http://arxiv.org/abs/math.MG/0406349}
  {\path{arXiv:math.MG/0406349}}, \href
  {http://dx.doi.org/10.1016/j.aim.2003.12.001}
  {\path{doi:10.1016/j.aim.2003.12.001}}.

\bibitem{MN-cotype}
M.~Mendel and A.~Naor.
\newblock Metric cotype.
\newblock {\em Ann. of Math. (2)}, 168(1):247--298, 2008.
\newblock Preliminary version in SODA~'06.
\newblock \href {http://arxiv.org/abs/math/0506201}
  {\path{arXiv:math/0506201}}.

\bibitem{Meyer}
P.-A. Meyer.
\newblock Transformations de {R}iesz pour les lois gaussiennes.
\newblock In {\em Seminar on probability, {XVIII}}, volume 1059 of {\em Lecture
  Notes in Math.}, pages 179--193. Springer, Berlin, 1984.

\bibitem{MS}
V.~D. Milman and G.~Schechtman.
\newblock {\em Asymptotic theory of finite-dimensional normed spaces}, volume
  1200 of {\em Lecture Notes in Mathematics}.
\newblock Springer-Verlag, Berlin, 1986.
\newblock With an appendix by M. Gromov.

\bibitem{NPSS06}
A.~Naor, Y.~Peres, O.~Schramm, and S.~Sheffield.
\newblock Markov chains in smooth {B}anach spaces and {G}romov-hyperbolic
  metric spaces.
\newblock {\em Duke Math. J.}, 134(1):165--197, 2006.
\newblock \href {http://arxiv.org/abs/math/0410422}
  {\path{arXiv:math/0410422}}, \href
  {http://dx.doi.org/10.1215/S0012-7094-06-13415-4}
  {\path{doi:10.1215/S0012-7094-06-13415-4}}.

\bibitem{NR-2005}
A.~Naor and Y.~Rabani.
\newblock Spectral inequalities on curved spaces.
\newblock Preprint, 2005.

\bibitem{NS-2004}
A.~Naor and L.~Silberman.
\newblock Poincar\'e inequalities, embeddings, and wild groups.
\newblock Preprint, 2004.

\bibitem{Ohta08}
S.~Ohta.
\newblock Markov type of {A}lexandrov spaces of nonnegative curvature.
\newblock To appear in Mathematika, 2008.
\newblock \href {http://arxiv.org/abs/0707.0102} {\path{arXiv:0707.0102}}.

\bibitem{Ozawa}
N.~Ozawa.
\newblock A note on non-amenability of {${\mathscr B}(l_p)$} for {$p=1,2$}.
\newblock {\em Internat. J. Math.}, 15(6):557--565, 2004.
\newblock \href {http://arxiv.org/abs/math.FA/0401122}
  {\path{arXiv:math.FA/0401122}}.

\bibitem{Pisier-martingales}
G.~Pisier.
\newblock Martingales with values in uniformly convex spaces.
\newblock {\em Israel J. Math.}, 20(3-4):326--350, 1975.
\newblock \href {http://dx.doi.org/10.1007/BF02760337}
  {\path{doi:10.1007/BF02760337}}.

\bibitem{pisier-79}
G.~Pisier.
\newblock Some applications of the complex interpolation method to {B}anach
  lattices.
\newblock {\em J. Analyse Math.}, 35:264--281, 1979.

\bibitem{Pisier-K-convex}
G.~Pisier.
\newblock Holomorphic semigroups and the geometry of {B}anach spaces.
\newblock {\em Ann. of Math. (2)}, 115(2):375--392, 1982.

\bibitem{pisier-2007}
G.~Pisier.
\newblock A remark on hypercontractive semigroups and operator ideals.
\newblock Preprint, 2007.
\newblock \href {http://arxiv.org/abs/0708.3423} {\path{arXiv:0708.3423}}.

\bibitem{pisier-2008}
G.~Pisier.
\newblock Complex interpolation between {H}ilbert, {B}anach and operator
  spaces.
\newblock Preprint, 2008.
\newblock \href {http://arxiv.org/abs/0802.0476} {\path{arXiv:0802.0476}}.

\bibitem{PX87}
G.~Pisier and Q.~H. Xu.
\newblock Random series in the real interpolation spaces between the spaces
  {$v\sb p$}.
\newblock In {\em Geometrical aspects of functional analysis (1985/86)}, volume
  1267 of {\em Lecture Notes in Math.}, pages 185--209. Springer, Berlin, 1987.

\bibitem{RTV06}
O.~Reingold, L.~Trevisan, and S.~P. Vadhan.
\newblock Pseudorandom walks on regular digraphs and the {RL} vs. {L} problem.
\newblock In J.~M. Kleinberg, editor, {\em STOC}, pages 457--466. ACM, 2006.
\newblock \href {http://dx.doi.org/10.1145/1132516.1132583}
  {\path{doi:10.1145/1132516.1132583}}.

\bibitem{RVW}
O.~Reingold, S.~Vadhan, and A.~Wigderson.
\newblock Entropy waves, the zig-zag graph product, and new constant-degree
  expanders.
\newblock {\em Ann. of Math.}, 155(1):157--187, 2002.
\newblock \href {http://arxiv.org/abs/math.CO/0406038}
  {\path{arXiv:math.CO/0406038}}.

\bibitem{Roe03}
J.~Roe.
\newblock {\em Lectures on coarse geometry}, volume~31 of {\em University
  Lecture Series}.
\newblock American Mathematical Society, Providence, RI, 2003.

\bibitem{RV05}
E.~Rozenman and S.~Vadhan.
\newblock \href{http://eccc.hpi-web.de/eccc-reports/2005/TR05-092}{Derandomized
  squaring of graphs}.
\newblock In {\em Approximation, randomization and combinatorial optimization},
  volume 3624 of {\em Lecture Notes in Comput. Sci.}, pages 436--447. Springer,
  Berlin, 2005.
\newblock \href {http://dx.doi.org/10.1007/11538462_37}
  {\path{doi:10.1007/11538462_37}}.

\end{thebibliography}

 \end{document}

By a standard computation, $\theta\asymp
t^{\frac{\pi}{2\phi}}$, where the implied constant depends only on
$a$ and $\phi$. The easiest way to see this fact to replace the
triangle $V$ by the sector $V'$ of angle $2\phi$, center at the
origin, and radius $2\sqrt{a^2+\pi^2}$ (equivalently, replace the
vertical segment $V_1$ by a circular arc centered at the origin).
One can easily verify that $\theta$ is (up to constant factors) the
harmonic measure of the arc in $\partial V'$. If we then apply the
mapping $z\mapsto z^{\frac{\pi}{2\phi}}$, then by conformal
invariance we reduce the problem to computing the probability that
Brownian motion starting at $t^{\frac{\pi}{2\phi}}$ leaves the half
disc with radius
$\left(2\sqrt{a^2+\pi^2}\right)^{\frac{\pi}{2\phi}}$ in the right
half of the plane at the circular arc on its boundary. This
computation is not difficult --- for example one can apply another
holomorphic mapping which sends this half circle to the vertical
strip bounded by the lines $\Re z=0$ and $\Re z =1$, where the half
circle is mapped to $\Re z =1$. For the resulting strip the comp